\documentclass[11pt]{amsart}

\usepackage[colorlinks = true,
            linkcolor = red,
            urlcolor  = magenta,
            citecolor = black,
            anchorcolor = blue]{hyperref}

\usepackage{color}
\usepackage[shortalphabetic,initials,nobysame]{amsrefs}
\usepackage{upgreek}
\usepackage{tikz}
\usepackage[applemac]{inputenc} 

\usetikzlibrary{arrows}
\usepackage[all]{xy}
\usepackage{graphicx}

\usepackage{wrapfig}
\usepackage{yfonts}
\usepackage{graphicx}
\usepackage{amssymb}
\usepackage{epstopdf}
\usepackage{mathrsfs}

\usepackage[mathcal]{eucal}

\usepackage{bm}

\let\oldmarginpar\marginpar
\renewcommand\marginpar[1]{\-\oldmarginpar[\raggedleft\footnotesize #1]%
{\raggedright\footnotesize #1}}

\renewcommand{\eprint}[1]{\href{https://arxiv.org/abs/#1}{#1}}

\DeclareMathOperator{\GL}{\mathrm{GL}}

\DeclareMathOperator{\SL}{\mathrm{SL}}

\DeclareMathOperator{\diag}{diag}

\DeclareMathOperator{\Lie}{Lie}

\DeclareMathOperator{\Tr}{Tr}

\newcommand{\fg}{\frak{g}}
\newcommand{\blam}{\bm{\lambda}}
\newcommand{\lam}{\lambda}

\newcommand{\barq}{\bar{q}}

\BibSpec{article}{%
    +{}  {\PrintAuthors}                {author}
    +{,} { \textit}                     {title}
     +{,} {}                            {note}
    +{.} { }                            {part}
    +{:} { \textit}                     {subtitle}
    +{,} { \PrintContributions}         {contribution}
    +{.} { \PrintPartials}              {partial}
    +{,} { }                            {journal}
    +{}  { \textbf}                     {volume}
    +{}  { \PrintDatePV}                {date}
    +{,} { \issuetext}                  {number}
    +{,} { \eprintpages}                {pages}
    +{,} { }                            {status}
    +{,} { \DOI}                   {doi}
    +{,} { \eprint}        {eprint}
      +{,} {\publisher}                {publisher}
    +{,} { \address}                   {address}
    +{}  { \parenthesize}               {language}
    +{}  { \PrintTranslation}           {translation}
    +{;} { \PrintReprint}               {reprint}
    +{.} {}                             {transition}
    +{}  {\SentenceSpace \PrintReviews} {review}
}

\newtheorem{Thm}{Theorem}[section]
\newtheorem{Lem}[Thm]{Lemma}
\newtheorem{Prop}[Thm]{Proposition}
\newtheorem{Cor}[Thm]{Corollary}

\theoremstyle{definition}
\newtheorem{Def}[Thm]{Definition}
\newtheorem*{thm:mainthm}{Theorem \ref{thm:mainthm}}
\theoremstyle{remark}
\newtheorem{Rem}[Thm]{Remark}

\newtheoremstyle{named}{}{}{\itshape}{}{\bfseries}{.}{.5em}{#1 #3}
\theoremstyle{named}

\def\C{\mathbb{C}}

\def\P{\mathbb{P}}

\def\fb{\mathfrak{b}}
\def\fh{\mathfrak{h}}
\def\g{\mathfrak{g}}
\def\Frenkel:2013uda{\mathfrak{h}}

\def\cF{\mathcal{F}}

\def\cL{\mathcal{L}}

\def\cV{\mathcal{V}}
\def\cW{\mathcal{W}}

\def\a{\alpha}

\def\bo{\textbf{o}}

\def\=>{\Longrightarrow}

\def\iff{\Longleftrightarrow}

\def\to{\longrightarrow}

\def\o+{\oplus}
\def\bo+{\bigoplus}

\def\<{\langle}
\def\>{\rangle}
\def\({\left(}
\def\){\right)}

\def\^{\wedge}
\def\+{\dagger}

\def\dd[#1,#2]{\frac{d#1}{d#2}}
\def\del[#1,#2]{\frac{\partial #1}{\partial #2}}
\def\over[#1]{\overline{#1}}
\def\vec[#1]{\overrightarrow{#1}}

\makeatletter
\def\mr@ignsp#1 {\ifx\:#1\@empty\else #1\expandafter\mr@ignsp\fi}%
\newcommand{\multiref}[1]{\begingroup
\xdef\mr@no@sparg{\expandafter\mr@ignsp#1 \: }%
\def\mr@comma{}%
\@for\mr@refs:=\mr@no@sparg\do{\mr@comma\def\mr@comma{,}\ref{\mr@refs}}%
\endgroup}
\makeatother

\newcommand{\hypref}[2]{\ifx\href\asklFrenkel:2013udaas #2\else\href{#1}{#2}\fi}

\tikzset{->-/.style={decoration={
  markings,
  mark=at position .5 with {\arrow{latex}}},postaction={decorate}}}
\tikzset{
    >=latex
    }

\newcommand{\wt}{\widetilde}

\newcommand{\nc}{\newcommand}
\nc{\on}{\operatorname}
\nc{\la}{\lambda}
\nc{\wh}{\widehat}
\nc{\ghat}{\wh\g}
\nc{\mb}{\mathbf}
\newcommand{\cha}{\check{\alpha}}
\nc{\tq}{\tilde{q}}

\begin{document}
\title{Opers on the projective line, Wronskian relations, and the Bethe Ansatz}

\author[T.J. Brinson]{Ty J. Brinson}
\address{
          Department of Mathematics, 
          Louisiana State University, 
          Baton Rouge, LA 70803, USA}

\author[D.S. Sage]{Daniel S. Sage}

\author[A.M. Zeitlin]{Anton M. Zeitlin}

\numberwithin{equation}{section}

\begin{abstract}
  It is well-known that the spectra of the Gaudin model may be
  described in terms of solutions of the Bethe Ansatz equations.  A
  conceptual explanation for the appearance of the Bethe Ansatz
  equations is provided by appropriate $G$-opers: $G$-connections on
  the projective line with extra structure.  In fact, solutions of the
  Bethe Ansatz equations are parameterized by an enhanced version of
  opers called Miura opers; here, the opers appearing have only
  regular singularities.  Feigin, Frenkel, Rybnikov, and Toledano
  Laredo have introduced an inhomogeneous version of the Gaudin model;
  this model incorporates an additional twist factor, which is an
  element of the Lie algebra of $G$.  They exhibited the Bethe Ansatz
  equations for this model and gave a geometric interpretation of the
  spectra in terms of opers with an irregular singularity.  In this
  paper, we consider a new approach to the study of the spectra of the
  inhomogeneous Gaudin model in terms of a further enhancement of
  opers called twisted Miura-Pl\"ucker opers and a certain system of
  nonlinear differential equations called the $qq$-system.  We show
  that there is a close relationship between solutions of the
  inhomogeneous Bethe Ansatz equations and polynomial solutions of the
  $qq$-system and use this fact to construct a bijection between the
  set of solutions of the inhomogeneous Bethe Ansatz equations and the
  set of nondegenerate twisted Miura-Pl\"ucker opers.  We further
  prove that as long as certain combinatorial conditions are
  satisfied, nondegenerate twisted Miura-Pl\"ucker opers are in fact
  Miura opers.


\end{abstract}

\maketitle

\setcounter{tocdepth}{1}
\tableofcontents


\section{Introduction}

The Bethe Ansatz is a classical approach to computing the spectra of
various quantum integrable systems, and in particular, spin chain
models.  This method is often very effective, but it is less easy to
understand conceptually the reason for this effectiveness.   The
Gaudin model is one context in which such an explanation is known.

Let $\fg$ be a simple complex Lie algebra with universal enveloping
algebra $U(\fg)$ and Langlands dual algebra ${}^L\fg$. In the Gaudin
model for $\fg$, one considers a family of mutually commuting elements
in $U(\fg)^{\otimes N}$ called Gaudin Hamiltonians, which depend
on a collection of distinct complex numbers $z_1,\dots,z_N$.  The
Bethe Ansatz provides a methods of constructing simultaneous
eigenvectors of the Gaudin Hamiltonians on modules such as
$V_{\blam}=\bigotimes_{i=1}^N V_{\lam_i}$, where $V_{\lam}$ is the
irreducible highest-weight module corresponding to the dominant
integral weight $\lam$.  One starts with the unique (up to scalar)
vector $|0\rangle\in V_{\blam}$ of highest weight $\sum\lam_i$; it is
a simultaneous eigenvector.  Given a set of distinct complex numbers
$w_1,\dots,w_m$ labeled by simple roots $\alpha_{k_j}$ (defined in
terms of fixed Cartan and Borel subalgebras $\mathfrak{h}\subset\fb_+$), one then
applies a certain order $m$ lowering operator with poles at the
$w_j$'s to $|0\rangle$.  If $\sum\lam_i-\sum \alpha_{k_j}$ is
dominant, then this vector is a highest weight vector in $V_{\blam}$
(and a simultaneous eigenvector of the Gaudin Hamiltonians) 
if and only if the Bethe Ansatz equations are satisfied:
\begin{equation}\label{e:Gaudinbethe} \sum_{i=1}^N \frac{\langle
    \lam_i,\cha_{k_j}\rangle}{w_j-z_i}- \sum_{s\ne j} \frac{\langle
    \a_{k_s},\cha_{k_j}\rangle}{w_j-w_s}=0, \quad j=1,\dots,m.
\end{equation}

In a series of papers~\cite{Feigin:1994in,Frenkel:2003qx,
  Frenkel:2004qy} , Frenkel and his collaborators provided a
conceptual explanation for this result.  They showed that the spectra
of the Gaudin model is encoded by certain connections with extra
structure associated to ${}^L\fg$ called \emph{opers}.  The opers
appearing here have regular singularities at $z_1,\dots,z_N$ and
$\infty$ and have trivial monodromy~\cite{Feigin:1994in,
  Frenkel:2004qy}.  These opers also have apparent singularities at
the $w_j$'s, and the Bethe Ansatz equations are precisely the
conditions for these singularities to be removable.  Moreover, this
approach allows one to give geometric meaning to solutions of the
Bethe Ansatz equations without assuming that
$\sum\lam_i-\sum \alpha_{k_j}$ is dominant.  In fact, they correspond
bijectively to enhanced versions of opers called (nondegenerate) Miura
opers.

More recently, Feigin, Frenkel, Rybnikov, and Toledano Laredo have
introduced an ``inhomogeneous'' version of the Gaudin
model~\cite{Feigin:2006xs,Rybnikov:2010} which involves an extra
``twist parameter'' $\chi\in\fh^*$.  In these papers, the authors have
given a similar geometric interpretation of the spectra in terms of
opers, but here, the regular singularity at $\infty$ is replaced by a
double pole with ``2-residue'' $-\chi$.  Moreover, they have shown
that the Bethe Ansatz equations for this model are given by
\begin{equation}\label{e:Gaudintwistbethe} \sum_{i=1}^N \frac{\langle
    \lam_i,\cha_{k_j}\rangle}{w_j-z_i}- \sum_{s\ne j} \frac{\langle
    \a_{k_s},\cha_{k_j}\rangle}{w_j-w_s}=\langle \chi,\cha_{k_j}\rangle, \quad j=1,\dots,m.
\end{equation}

In this paper, we consider a new approach to the study of the spectra
of the inhomogeneous Gaudin model in terms of \emph{twisted Miura opers}
and a certain system of nonlinear differential equations called the
\emph{$qq$-system}.  As we will see, there is a close relationship
between solutions of the inhomogeneous Bethe Ansatz equations
\eqref{e:Gaudintwistbethe} and polynomial solutions of the
$qq$-system.  We will use this fact to construct a bijection
between the set of solutions of the
inhomogeneous Bethe Ansatz equations and the set of ``nondegenerate'' twisted
Miura opers.

Since we will be primarily concerned with opers, it will be convenient
to switch the roles of $\fg$ and ${}^L\fg$.  From now on, we consider
the Gaudin model for ${}^L\fg$, which will correspond to appropriate
$G$-opers, where $G$ is the simply connected group with Lie algebra
$\fg$.  The twist parameter may now be viewed as an element
$Z\in\fh$.\footnote{For much of the paper, we will in fact allow $Z$
  to be an element of a fixed Borel subalgebra $\fb_+$.}

Let $H$ be the maximal torus with Lie algebra $\fh$, and let $B_+$ and
$B_-$ be two opposite Borel subgroups containing $H$.  Roughly
speaking, an oper is a triple $(\cF_G,\nabla,\cF_{B_-})$, where
$\cF_G$ is a principal $G$-bundle on $\P^1$ endowed with a meromorphic
connection $\nabla$ and $\cF_{B_-}$ is a reduction of structure of the
bundle to $B_-$ and $\nabla$ is a meromorphic connection on $\cF_G$
which satisfies a certain genericity condition with respect to
$\nabla$.  A Miura oper is an oper together with an addition reduction
of structure $\cF_{B_+}$ to the opposite Borel subgroup which is
preserved by $\nabla$.  We now consider Miura opers whose underlying
connection has regular singularities away from infinity, is
monodromy-free, and is ``$Z$-twisted''.  It turns out that the set of
twisted Miura opers with the same underlying oper is a subvariety of
the flag manifold called the Springer fiber over $Z$.  Finally, given
a Miura oper, we construct a family of Miura $\GL(2)$-opers
parameterized by the fundamental weight.  The underlying Miura oper is
called a \emph{$Z$-twisted Miura-Pl\"ucker $G$-oper} if the zero
monodromy and $Z$-twistedness conditions hold on this family of Miura
$\GL(2)$-opers and not necessarily on the $G$-oper itself.

In this paper, we show that solutions of the $Z$-twisted Bethe Ansatz
equations for ${}^L \fg$ are parameterized by nondegenerate
$Z$-twisted Miura $G$-opers.  In order to accomplish this, we
introduce a system of differential equations called the $qq$-system
associated to $G$, the regular singularities $z_j$, and the twist
parameter $Z$.  This is a nonlinear system on a collection of rational
functions $\{q^i_+(z),q^i_-(z)\}_{i\in\Delta}$, indexed by the set of
simple roots $\Delta$, which determine relations satisfied by the
Wronskians $W(q^i_+(z),q^i_-(z))$.  We first construct a surjection
from nondegenerate polynomial solutions of the $qq$-system for $Z$ to
nondegenerate $Z$-twisted Miura-Pl\"ucker opers
(Corollary~\ref{MPsurj}).  (In fact, we give a bijection between these
solutions and ``$Z$-twisted Miura-Pl\"ucker data''
(Theorem~\ref{inj}).)  Next, we prove that there is a surjective map
from these polynomial solutions to solutions of the Bethe Ansatz
equation (Theorem~\ref{BAE}.  We show that the fibers of these
surjections coincide, thereby obtaining a one-to-one correspondence
between nondegenerate $Z$-twisted Miura-Pl\"ucker opers and solutions
of the Bethe Ansatz equations (Theorem~\ref{MPbethe}).

We then introduce the crucial technical tool of \emph{B\"acklund
  transformations}: transformations on twisted Miura-Pl\"ucker opers
associated to Weyl group elements.  These transformations change not
only the Miura-Pl\"ucker oper, but also the twist factor.  We use
B\"acklund transformations to show that, as long as certain
combinatorial conditions are satisfied, nondegenerate twisted
Miura-Pl\"ucker opers are in fact Miura opers
(Theorem~\ref{thm:MPisM}).  As a corollary, we obtain the following
important theorem (Theorem~\ref{Mbethe}): under appropriate
combinatorial hypotheses, there is a bijection between
nondegenerate $Z$-twisted Miura opers and solutions to the Bethe
Ansatz equations.

Our approach to this problem was inspired by recent work of Frenkel,
Koroteev, and two of the authors on a $q$-deformation of the
correspondence between opers and the spectra of the Gaudin
model~\cite{KSZ,Frenkel:2020}.  These papers relate solutions of the
Bethe Ansatz for the XXZ-model to certain $q$-difference equation
versions of opers called twisted Miura-Pl\"ucker
$(G,q)$-opers.
The
correspondence goes through the intermediary of the ``$QQ$-system'': a
system of $q$-difference equations involving quantum Wronskians, which
was introduced by Masoero, Raimondo, and
Valeri~\cite{Masoero_2016,Masoero_2016_SL} (see also
\cite{Frenkel:ac}).  However, we observe that our present results go
beyond what is known about the XXZ model. In particular, the results of
\cite{KSZ,Frenkel:2020} are limited to the case when the twist
parameter is regular semisimple.

\subsection*{Acknowledgements} We are grateful to Edward Frenkel for
his valuable comments.  A.M.Z. is partially supported by Simons
Collaboration Grant 578501 and NSF grant
DMS-2203823.  D.S.S is partially supported
by Simons Collaboration Grant 637367.

\section{$G$-opers with regular singularities}
\label{Sec:DefinitionsOpers}

\subsection{Notation and group-theoretic background}    \label{regsing} 

Let $G$ be a connected, simply connected, simple algebraic group of
rank $r$ over $\mathbb{C}$.  We fix a Borel subgroup $B_-$ with
unipotent radical $N_-=[B_-,B_-]$ and a maximal torus $H\subset B_-$.
Let $B_+$ be the opposite Borel subgroup containing $H$ and
$N_+=[B_+,B_+]$.  Let $\{ \alpha_1,\dots,\alpha_r \}$ be the set of
positive simple roots for the pair $H\subset B_+$.  Let
$\{ \check\alpha_1,\dots,\check\alpha_r \}$ be the corresponding
coroots; the elements of the Cartan matrix of the Lie algebra $\fg$ of
$G$ are given by $a_{ij}=\langle
\alpha_j,\check{\alpha}_i\rangle$. The Lie algebra $\fg$ has Chevalley
generators $\{e_i, f_i, \check{\alpha}_i\}_{i=1, \dots, r}$, so that
$\fb_-=\Lie(B_-)$ is generated by the $f_i$'s and the
$\check{\alpha}_i$'s and $\fb_+=\Lie(B_+)$ is generated by the $e_i$'s
and the $\check{\alpha}_i$'s.  Similarly the Lie algebra
$\mathfrak{n}_-=\Lie(N_-)$ is generated by the $f_i$'s and
$\mathfrak{n}_+=\Lie(N_+)$ is generated by the $e_i$'s.  Let
$\omega_1,\dots\omega_r$ be the fundamental weights, defined by
$\langle \omega_i, \check{\alpha}_j\rangle=\delta_{ij}$.

Let $W=N(H)/H$ be the Weyl group of $G$.  Let $W$ be the Weyl group of
$G$.  For each $i$, we let $s_i\in W$ be the simple reflection corresponding to
$\alpha_i$. We also let $w_0$ be the longest element of $W$, so
that $B_+=w_0(B_-)$.

Recall that for any Borel subgroup $B$, the group $G$ is partitioned
into  Bruhat cells $BwB$ indexed by elements of $W$.  Here, one chooses
some maximal torus $T\subset B$ and sets $BwB=BnB$, where $n$ is
any lift of $w\in N(T)/T\cong W$.  Since we defined  $W$ in terms
of $H$, it is not immediately clear that this process makes sense.
However, an argument involving the ``abstract Cartan algebra'' (see
for example
~\cite[\S 3.1.22]{CG}) shows that the 
Bruhat cells are well-defined.  We refer the reader to \S 2.1 of
\cite{Frenkel:2020} for the details.

\subsection{Meromorphic $G$-opers}

We now define meromorphic $G$-opers.  While the definitions below may
be extended easily to an arbitrary smooth curve, we will restrict
ourselves to the case of $\P^1$.

Let $\cF_G$ be a principal $G$-bundle on $\P^1$ endowed with a
connection $\nabla$.  This connection is automatically flat.  Let
$\cF_{B_-}$ be a reduction of $\cF_G$ to the Borel subgroup $B_-$.
If $\nabla'$ is any connection which preserves $\cF_{B_-}$, then
$\nabla-\nabla'$ induces a well-defined one-form on $\P^1$ with values
in the associated bundle
$(\mathfrak{g}/\mathfrak{b}_-)_{\cF_{B_-}}$. We denote this 1-form by
$\nabla/\cF_{B_-}$.

Following \cite{Beilinson:2005}, we will define a $G$-oper as a
$G$-connection $(\cF_G,\nabla)$ together with a  reduction $\cF_{B_-}$
of the $G$-bundle
to the Borel subgroup $B_-$; this reduction is not preserved by the 
connection but instead satisfies a special ``transversality
condition''  defined in terms of the 1-form $\nabla/\cF_{B_-}$.

To define this transversality condition, let $\bf{O}\in
[\mathfrak{n}_-,
\mathfrak{n}_-]^{\perp}/\mathfrak{b}_-\in\mathfrak{g}/\mathfrak{b}_-$
be the open  $B_-$-orbit consisting of vectors stabilized by $N_-$ and
such that all of the simple root components with
respect to the adjoint action of $B_-/N_-$, are non-zero.  Here, the
orthogonal complement is taken with respect to the Killing form.

\begin{Def}    \label{op}
  A meromorphic $G$-{\em oper}  on
  $\mathbb{P}^1$ is a triple $(\cF_G,\nabla,\cF_{B_-})$, where $\cF_G$
  is a principal $G$-bundle on $\P^1$ equipped with a meromorphic
  connection $\nabla$ and $\mathcal{F}_{B_-}$ is a reduction of $\cF_G$
  to $B_-$ satisfying the following condition: there exists a Zariski
  open dense subset $U \subset \P^1$ together with a trivialization
  $\imath_{B_-}$ of $\mathcal{F}_{B_-}$ such that the restriction of the
  1-form $\nabla/\cF_{B_-}$ to $U$, written as an element of $\mathfrak{g}/\mathfrak{b}_-(z)$, belongs to $\mathbf{O}(z)$.
\end{Def}

Note that this property does not depend on the choice of
trivialization.

In terms of the particular trivialization $\imath_{B_-}$, the
underlying connection of the $G$-oper can be written concretely as
\begin{equation}    \label{op1}
\nabla=\partial_z+\sum^r_{i=1}\phi_i(z)e_i+b(z)
\end{equation}
where $\phi_i(z) \in\C(z)$ and  $b(z)\in \mathfrak{b}_-(z)$ are
regular on $U$ and moreover $\phi_i(z)$ has no zeros in $U$.

\subsection{Miura opers}

We will also need the notion of a Miura oper introduced in
\cite{Frenkel:2003qx,Frenkel:2004qy}. This is an oper together with a
reduction of the underlying $G$-bundle to the opposite Borel subgroup
that is preserved by the oper connection.

\begin{Def}    \label{Miura}
  A {\em Miura $G$-oper} on $\mathbb{P}^1$ is a quadruple
 $(\cF_G,\nabla,\cF_{B_-},\cF_{B_+})$, where $(\cF_G,\nabla,\cF_{B_-})$ is a
  meromorphic $G$-oper on $\P^1$ and $\cF_{B_+}$ is a reduction of
  the $G$-bundle $\cF_G$ to $B_+$ that is preserved by the
  connection $\nabla$.
\end{Def}

Given a Miura $G$-oper, we refer to the $G$-oper obtained by forgetting
$\cF_{B_+}$ the underlying $G$-oper.

We next need to consider the relative position of the two reductions
over any $x\in\P^1$.  This relative position will be an element of the
Weyl group.  To define this, first
note that the fiber
$\cF_{G,x}$ of $\cF_G$ at $x$ is a $G$-torsor with reductions
$\cF_{B_-,x}$ and $\cF_{B_+,x}$ to $B_-$ and $B_+$
respectively.   Under this
isomorphism, $\cF_{B_-,x}$ gets identified with $gB_- \subset G$ and
$\cF_{B_+,x}$ with $hB_+$ for some $g,h\in G$. The quotient $g^{-1}h$
specifies an element of
the double coset space $B_-\backslash G/B_+$.  The Bruhat
decomposition gives a bijection
between this spaces and the Weyl group, so we obtain a well-defined
element of $G$.

We say that $\cF_{B_-}$ and $\cF_{B_+}$ have  {\em generic
  relative position} at $x\in\P^1$ if the relative position is the
identity element of $W$.  More concretely, this mean that the quotient
$g^{-1}h$ belongs to the open dense Bruhat cell $B_-B_+ \subset
G$.

The following result was proved in
\cite{Frenkel:2003qx,Frenkel:2004qy}.  It will be convenient to give a
different proof here.

\begin{Thm}    \label{gen rel pos}
  For any Miura $G$-oper on $\mathbb{P}^1$, there exists an open
  dense subset $V \subset \P^1$ such that the reductions $\cF_{B_-}$
  and $\cF_{B_+}$ are in generic relative position for all $x \in V$.
\end{Thm}

\begin{proof}
  Let $U$ be a Zariski open dense subset on $\P^1$ as in Definition
  \ref{op}. Choosing a trivialization $\imath_{B_-}$ of $\cF_G$ on $U$, we can write the connection $\nabla$ in the form
  \eqref{op1}.  On the other hand, using the $B_+$-reduction
  $\cF_{B_+}$, we can choose another trivialization of $\cF_G$ on $U$
  such that the connection in this gauge is preserved by $\nabla$.  In
  other words, there exists $g(z) \in G(z)$ such that
\begin{equation}    \label{connecting}
g(z)\partial_z g^{-1}(z) +g(z)(\sum^r_{i=1}\phi_i(z)e_i+b(z))g^{-1}(z)\in \mathfrak{b}_+(z) 
\end{equation}
This means that the relative position of the two reductions is determined
by $g^{-1}(z)$.   It thus suffices to show that $g^{-1}(z)\in
B_-(z)B_+(z)$ or equivalently, $$g(z)\in B_+(z) B_-(z)=B_+(z) N_-(z).$$

By the Bruhat decomposition, we know that $g(z)\in B_+(z) w N_-(z)$
for some $w\in W$, say  $g(z) = b_+(z) w n_-(z)$ for
some $b_+(z) \in B_+(z), n_-(z) \in N_-(z)$.
Substituting this into \eqref{connecting} and simplifying gives
\begin{equation} \label{miurafinal}
n_-(z)\partial_zn_-^{-1}(z)+n_-(z)(\sum^r_{i=1}\phi_i(z)e_i+b(z))n_-^{-1}(z)=
\sum^r_{i=1}\phi_i(z)e_i+\tilde{b}(z)\in w^{-1}\mathfrak{b}_+(z)w,
\end{equation}
where $\tilde{b}(z)\in \mathfrak{b}_{-}(z)$.  It is well-known that
$w^{-1}\mathfrak{b}_+w=\mathfrak{h}+(\mathfrak{n}_-\cap
w^{-1}\mathfrak{b}_+w))+(\mathfrak{n}_+\cap w^{-1}\mathfrak{b}_+w)$.
Since the 
strictly upper triangular component of the expression in
\eqref{miurafinal} is $\sum^r_{i=1}\phi_i(z)e_i$, we conclude that
$\phi_i(z)e_i\in w^{-1}\mathfrak{b}_+w$ for all $i$.  This means that
$w$ preserves the set of simple roots, i.e., $w=1$.

\end{proof}

\begin{Cor}    \label{gen rel pos1}
For any Miura $G$-oper on $\mathbb{P}^1$, there exists a
trivialization of the underlying $G$-bundle $\cF_G$ on an open
dense subset of $\P^1$ for which the oper connection has the form
\begin{equation}    \label{genmiura}
\nabla=\partial_z+\sum^r_{i=1}g_i(z)\check{\alpha}_i+\sum^r_{i=1}{\phi_i(z)}e_i,
\end{equation}
where $g_i(z), \phi_i(z)\in \mathbb{C}(z)$.
\end{Cor}

\begin{proof}

  The previous theorem shows that $w=1$ in \eqref{miurafinal}, so 
  there exists a gauge transformation
  $n_-(z)$ which takes the explicit form of the connection $\nabla=\partial_z+\sum^r_{i=1}\phi_i(z)e_i+b(z)$ into
\begin{equation}
n_-(z)\partial_zn_-^{-1}(z)+n_-(z)(\sum^r_{i=1}\phi_i(z)e_i+b(z))n_-^{-1}(z)=
\sum^r_{i=1}\phi_i(z)e_i+\tilde{b}(z)\in \mathfrak{b}_+(z)
\end{equation}
where $\tilde{b}(z)\in \mathfrak{b}_-(z)$.  This implies that  
$\tilde{b}(z)\in \mathfrak{h}(z)$, and the statement follows by decomposing $\tilde{b}(z)$ with
respect to the simple coroots.
\end{proof}

\subsection{Opers and Miura opers with regular singularities}

Let $\Lambda_1(z),\dots, \Lambda_r(z)$ be a collection of
nonzero polynomials.

\begin{Def} \label{d:regsing} A $G$-{\em oper with regular
    singularities determined by $\Lambda_1(z),\dots, \Lambda_r(z)$}
  is an oper on $\P^1$ whose connection \eqref{op} may be written in
  the form
\begin{equation}    \label{Lambda}
\nabla=\partial_z+\sum^r_{i=1}\Lambda_i(z)e_i+b(z), \qquad
b(z)\in \mathfrak{b}_-(z).
\end{equation}
\end{Def}

We will assume without loss of generality that the $\Lambda_i$'s are
monic, since this can always be arranged by a constant gauge change by an element of $H$.
Let $\{z^i_1,\dots,z^i_{N_i}\}$ be the set of distinct roots of the 
$\Lambda_i$'s.  To each $z^i_k$, we associate the integral coweight $\check{\lambda}_k$ via
\begin{equation}\label{lambdaroots}
\Lambda_i(z)=\prod^{N_i}_{k=1}(z-z^i_k)^{\langle  {\alpha}_i,\check{\lambda}_k\rangle}.
\end{equation}

\begin{Def}    \label{MiuraRS}
  {\em A Miura $G$-oper with regular singularities determined by
the polynomials $\Lambda_1(z),\dots, \Lambda_r(z)$} is a Miura
  $G$-oper whose underlying oper has
regular singularities determined by the $\Lambda_i(z)$'s.
\end{Def}

The following theorem is immediate from Corollary \ref{gen rel pos1}. 

\begin{Thm}    \label{gen elt1}
For every Miura $G$-oper with regular singularities determined by
the polynomials $\Lambda_1(z),\dots, \Lambda_r(z)$, the underlying
connection can be written in the form:
\begin{equation}    \label{form of A}
\nabla=\partial_z+\sum^r_{i=1}\Lambda_i(z)e_i+\sum^r_{i=1}g_i(z)\check{\alpha}_i, 
\end{equation}
where $g_i(z) \in \C(z).$
\end{Thm}

For the rest of the paper, all opers and Miura opers will have regular
singularities with respect to the fixed collection of monic
polynomials $\Lambda_1(z),\dots, \Lambda_r(z)$.

\subsection{$Z$-twisted opers}    \label{Sec:Ztw}

We will primarily be interested in (Miura) opers whose underlying
connection is gauge
equivalent to a constant element of $\mathfrak{g}$.

\begin{Def}    \label{Ztwoper}
  A {\em $Z$-twisted $G$-oper} on $\mathbb{P}^1$ is a $G$-oper
  that is equivalent to the constant element $Z \in \mathfrak{g} \subset \mathfrak{g}(z)$
  under the gauge action of $G(z)$.
\end{Def}

Concretely,  if the matrix form of the oper connection in a particular
trivialization is given by $\nabla=\partial_z+A(z)$, then there exists 
  $g(z) \in  G(z)$ such that
\begin{equation}    \label{Ag}
A(z)=g(z)\partial_z g^{-1}(z)+g(z)Z g(z)^{-1}.
\end{equation}

\begin{Rem} 
Note that for $Z\ne 0$, the constant connection $\partial_z+Z$ has a
double pole at $\infty$, so $Z$-twisted opers are the same as the
opers with a double pole at $\infty$ considered in
~\cite{Feigin:2006xs,Rybnikov:2010}.
\end{Rem}

To define $Z$-twisted Miura opers, we will assume that
$Z\in\mathfrak{b}_+$.  We introduce the notation
\begin{equation}    \label{Z}
Z =  Z^H+\sum^r_{i=1}c_ie_i+n, \qquad Z^H =  \sum^r_{i=1}\zeta_i\check\alpha_i,\qquad 
\zeta_i, c_i \in \mathbb{C}, \qquad n\in [\mathfrak{n}_+, \mathfrak{n}_+].
\end{equation}

\begin{Def}    \label{ZtwMiura}
A {\em $Z$-twisted Miura $G$-oper} is a Miura $G$-oper on
$\mathbb{P}^1$ that is equivalent to the constant element $Z \in \mathfrak{b}_+
\subset \mathfrak{b}_+(z)$ under the gauge action of $B_+(z)$, i.e.,
there exists $v(z)\in B_+(z)$ such that the matrix of the oper
connection is given by
\begin{equation}    \label{gaugeA}
A(z)=v(z)\partial_z v^{-1}(z)+v(z)Z v(z)^{-1}.
\end{equation}
\end{Def}

For untwisted opers, there is a full flag variety $G/B_+$ of
associated  
Miura opers. For twisted opers, we must introduce certain closed
subvarieties of the flag manifold of the form $(G/B_+)_Z=\{gB_+\mid
g^{-1}Zg\in\fb_+\}$; these varieties are called
\emph{Springer fibers}.  Springer fibers play an important role in
representation theory.  (See, for example, Chapter 3 of \cite{CG}.)  For
$\SL(n)$ (or $\GL(n)$),  a Springer fiber may be viewed as the space of complete
flags in $\C^n$ preserved by a fixed endomorphism.

\begin{Prop}  The map from Miura
$Z$-twisted opers to $Z$-twisted opers is a fiber bundle with fiber
$(G/B_+)_Z$.
\end{Prop}

\begin{proof}  Since the underlying connection of a $Z$-twisted oper is isomorphic to the 
  connection $\partial_z+Z$, a Miura structure on such an oper is equivalent
  to a $B_+$-reduction that is preserved by $\partial_z+Z$.  This is
  determined by a Borel subalgebra of $\fg$ that contains $Z$.  The
  flag variety may be identified with the space of Borel subalgebras
  via $gB\mapsto g\fb_+g^{-1}$, and the condition $Z\in g\fb_+g^{-1}$
  is equivalent to $gB\in (G/B_+)_Z$.
\end{proof}

\subsection{The associated Cartan connection}    \label{Hconn}

Consider a Miura $G$-oper with regular singularities determined by
polynomials $\Lambda_1(z),\dots, \Lambda_r(z)$. By Theorem
\ref{gen elt1}, the underlying $G$-connection can be written in
the form \eqref{form of A}. Since it preserves the $B_+$-bundle
$\cF_{B_+}$ that is part of the data of the Miura $G$-oper, it may be viewed as a meromorphic
$B_+$-connection on $\P^1$. Taking the quotient of $\cF_{B_+}$ by
$N_+ = [B_+,B_+]$ and using the fact that $B/N_+ \simeq H$, we obtain
an $H$-bundle $\cF_{B_+}/N_+$ endowed with an
$H$-connection, which we denote by $\nabla^H=\partial_z+A^H(z)$. According to formula
\eqref{form of A}, it is given by the formula
\begin{equation}    \label{AH}
A^H(z)=\sum^r_{i=1} g_i(z){\check{\alpha}_i}.
\end{equation}
We call $\nabla^H(z)=\partial_z+A^H(z)$ the \emph{associated Cartan connection} of the
Miura oper.

Now, if our Miura oper is $Z$-twisted, then we also have $A(z)=v(z)\partial_z v^{-1}(z)+v(z)Z v(z)^{-1}$, where
$v(z)\in B_+(z)$.  Since $v(z)$ can be written as
\begin{equation}    \label{vz}
v(z)=
\left(\prod_i y_i(z)^{\check{\alpha}_i}\right) n(z), \qquad n(z)\in N_+(z), \quad
y_i(z) \in \C(z)^\times,
\end{equation}
the Cartan connection $\nabla^H(z)=\partial_z+A^H(z)$ has the form:
\begin{equation}    \label{AH1}
A^H(z)=\sum^r_{i=1}
(\zeta_i -y_i(z)^{-1}\partial_zy_i(z))\check{\alpha}_i,
\end{equation}
with the $\zeta_i$'s defined in \eqref{Z}.  
We will refer to $\nabla^H(z)$ as a $Z$-{\em twisted Cartan connection}. This formula shows that $\nabla^H(z)$ is completely
determined by $Z^H$, i.e., the diagonal part of $Z$, and the rational functions $y_i(z)$. Indeed,
comparing this equation with \eqref{AH} gives
\begin{equation}    \label{giyi}
g_i(z)=\zeta_i -y_i(z)^{-1}\partial_zy_i(z)
\end{equation}

It is now easy to see that $\nabla^H(z)$ determines the $y_i(z)$'s uniquely
up to scalar. 

\section{Nondegenerate Miura-Pl\"ucker
  opers}    \label{Sec:nondegMiura}

Our main goal is to link Miura opers to solutions of a certain system
of equations which we will call the classical $qq$-system, which is in
turn related to the system of Bethe Ansatz equations for the Gaudin
model. We accomplish this in two steps.  First, we introduce the
notion of a $Z$-twisted Miura-Pl\"ucker $G$-oper.  We associate to a
Miura $G$-oper a collection of Miura $\GL(2)$-opers indexed by the
fundamental weights of $G$.  A $Z$-twisted Miura-Pl\"ucker oper is a
Miura oper where the $Z$-twistedness condition is replaced by a
slightly weaker condition imposed on these $\GL(2)$-opers.  Second, we
will restrict attention to opers satisfying certain nondegeneracy
conditions defined in terms of the corresponding Cartan connection.

\subsection{The associated Miura $\GL(2)$-opers}    \label{rank2}

In this section, we associate to a Miura $G$-oper with regular
singularities a collection of Miura $\GL(2)$-opers indexed by the
fundamental weights.

Let $V_i$ be the irreducible representation of $G$ with highest weight
given by the fundamental weight
$\omega_i$.  Let $L_i\subset V_i$ be the $B_+$-stable line consisting
of highest weight vectors.  If we choose a nonzero element
$\nu_{\omega_i}$ in $L_i$, then the subspace of $V_i$ of weight $\omega_i-\alpha_i$
is one-dimensional and is spanned by $f_i \cdot
\nu_{\omega_i}$.  Therefore, the two-dimensional subspace $W_i$ of
$V_i$ spanned by the weight vectors $\nu_{\omega_i}$ and $f_i \cdot
\nu_{\omega_i}$ is a $B_+$-invariant subspace of $V_i$.

Now, let $(\cF_G,\nabla,\cF_{B_-},\cF_{B_+})$ be a Miura $G$-oper with
regular singularities determined by polynomials
$\Lambda_1(z),\dots, \Lambda_r(z)$ as in Definition
\ref{MiuraRS}. Recall that $\cF_{B_+}$ is a $B_+$-reduction of a
$G$-bundle $\cF_G$ on $\P^1$ preserved by the $G$-connection
$\nabla$. Therefore for each $i$, the vector bundle
$$
\cV_i = \cF_{B_+} \times_{B_+} V_i = \cF_G \times_{G}
V_i
$$
associated to $V_i$ contains a rank two
subbundle
$$
\cW_i = \cF_{B_+} \times_{B_+} W_i
$$
associated to $W_i \subset V_i$, and $\cW_i$ in turn contains a line
subbundle
$$
\cL_i = \cF_{B_+} \times_{B_+} L_i
$$
associated to $L_i \subset W_i$.

Denote by $\phi_i(\nabla)$ the connection on the vector bundle $\cV_i$
(or equivalently, the $\GL(V_i)$-connection) corresponding to the
above Miura oper connection $\nabla$. Since $\nabla$ preserves
$\cF_{B_+}$, we see that $\phi_i(\nabla)$ preserves the subbundles
$\cL_i$ and $\cW_i$ of $\cV_i$. Denote by $\nabla_i$ the corresponding
connection on the rank 2 bundle $\cW_i$.

Trivialize $\cF_{B_+}$ on a Zariski open subset of $\P^1$ so
that $\nabla$ has the form \eqref{form of A} with respect to this
trivialization. This trivializes the
bundles $\cV_i$, $\cW_i$, and $\cL_i$ as well, so that the connection
$\nabla_i(z)$ can be expressed in terms of  a $2 \times 2$ matrix whose entries are in $\C(z)$.

A direct computation using formula \eqref{form of A} yields the
following result.

\begin{Lem}    \label{2flagthm}
We have
\begin{equation}    \label{2flagformula}
\nabla_i(z)=\partial_z+
\begin{pmatrix}
  g_i(z)\\
&\\  
  0 & -g_i(z)- \sum_{k\neq i}a_{ki}g_k(z)
 \end{pmatrix},
\end{equation}
\end{Lem}

Using the trivialization of $\cW_i$ in which $\nabla_i(z)$ has this
form, we can decompose $\cW_i$ as the direct sum of two
line subbundles. The first is $\cL_i$, generated by the basis vector
$\begin{pmatrix} 1 \\ 0 \end{pmatrix}$. The second, which we denote by
$\wt\cL_i$, is generated by the basis vector $\begin{pmatrix} 0 \\
  1 \end{pmatrix}$. The subbundle $\cL_i$ is $\nabla_i$-invariant,
whereas $\nabla_i$ satisfies the following \emph{$\GL(2)$-oper
  condition} with respect to $\wt\cL_i$.

\begin{Def}    \label{GL2}
  A \emph{$\GL(2)$-oper} on $\P^1$ is a triple $(\cW,\nabla,\wt\cL)$,
  where $\cW$ is a rank 2 bundle on $\P^1$, $\nabla: \cW \to \cW\otimes K$ is a
  meromorphic connection on $\cW$, $K$ is the canonical bundle on $\P^1$, and $\wt\cL$ is a line
  subbundle of $\cW$ such that the induced map $\bar{\nabla}:\wt\cL \to
  (\cW/\wt\cL)\otimes K$ is an isomorphism on a Zariski open dense subset of
  $\P^1$.

  A Miura \emph{$\GL(2)$-oper} on $\P^1$ is a quadruple
  $(\cW,\nabla,\wt\cL,\cL)$, where $(\cW,\nabla,\wt\cL)$ is a $\GL(2)$-oper
  and $\cL$ is an $\nabla$-invariant line subbundle of $\cW$.
\end{Def}

Using this definition, one obtains an alternative definition of
(Miura) $\SL(2)$-opers: they are the (Miura) $\GL(2)$-opers defined by
the above triples (resp. quadruples) satisfying the additional
property that in some trivialization on a Zariski-open dense subset of
$\P^1$, the trace of the matrix of the connection is $0$.

Our quadruple $(\cW_i,\nabla,\wt\cL_i,\cL_i)$ is clearly a Miura
$\GL(2)$-oper. It is not clear whether it is an $\SL(2)$-oper
because the trace of the matrix in \eqref{2flagformula}  is not
necessarily $0$.

We now make the further assumption that our Miura $G$-oper 
$(\cF_G,\nabla,\cF_{B_-}$, $\cF_{B_+})$ with regular singularities is
$Z$-twisted (see Definition \ref{ZtwMiura}).  Recall that this implies
that the associated Cartan connection $\nabla^H(z)$ has the form
\eqref{AH1}:
\begin{equation}    \label{AH2}
\nabla^H(z)=\prod_i
y_i(z)^{\check{\alpha}_i} \; (\partial_z+Z^H) \; \prod_i y_i(z)^{-\check{\alpha}_i},
\qquad y_i(z) \in \C(z).
\end{equation}
We claim that for $Z$-twisted Miura opers, there exists
another trivialization of $\cW_i$ in which the connection matrix of $\nabla_i$
has constant (though not necessarily zero) trace. This will be a
particularly convenient gauge for $\nabla_i$.

To prove the claim, let $A_i(z)$ denote the matrix in
\eqref{2flagformula}, and apply the gauge transformation by the
diagonal matrix

$$
 u(z)= \begin{pmatrix} 1 & 0 \\ 0 & \prod_{j\neq i} y_j(z)^{a_{ji}}
\end{pmatrix}.
 $$
This gives
\begin{equation}\label{cano}
\wt{\nabla}_i(z) = u(z)\nabla_i(z)u^{-1}(z)=\partial_z+\begin{pmatrix}
  \zeta_i-y_i(z)^{-1}\partial_z y_i(z)&\rho_i(z)\\
&\\  
  0 & - \sum_{k\neq i}a_{ki}\zeta_k-\zeta_i+y_i(z)^{-1}\partial_z y_i(z))
 \end{pmatrix},
 \end{equation}
 where
 \begin{equation}\label{ri}
 \rho_i(z)=\Lambda_i(z)\prod_{k\neq i}y_k(z)^{-a_{ki}}.
 \end{equation}
 Since $a_{ij} \leq 0$ for $i \neq j$, $\rho_i(z)$
 is a polynomial if all $y_j(z)$'s are polynomials.

Let $G_i\cong \SL(2)$ be the subgroup of $G$ corresponding to the
$\mathfrak{sl}(2)$-triple spanned by $\{e_i, f_i, \check{\alpha}_i\}$.
Note that the group $G_i$ preserves $W_i$. Consider the Miura
$G_i$-oper $(\cW_i,{\hat \nabla}_i,$ $\wt\cL_i,\cL_i)$ with $\wt\cL_i =
\on{span} \left\{ \begin{pmatrix} 0 \\ 1 \end{pmatrix} \right\}$,
$\cL_i = \on{span} \left\{ \begin{pmatrix} 1 \\ 0 \end{pmatrix}
\right\}$,
\begin{equation}\label{Bi}
\hat{\nabla}_i =\partial_z+g_i\check\alpha_i+\rho_i(z)e_i=
\begin{pmatrix}
  \zeta_i-y_i(z)^{-1}\partial_z y_i(z)&\rho_i(z)\\
&\\  
  0 & -\zeta_i+y_i(z)^{-1}\partial_z y_i(z))
 \end{pmatrix},
 \end{equation}
We can now express the connection $\wt{\nabla}_i(z)$ as the sum of an
$\SL(2)$-connection and a constant diagonal matrix:

\begin{align}    \label{tilde calig}
\wt{\nabla}_i(z) &= \begin{pmatrix} 1 & 0 \\ 0 & \sum_{j\ne i}
 -a_{ji}\zeta_j
\end{pmatrix} +{\hat \nabla}_i(z) \\    \label{Ai}
&= \partial_z+\begin{pmatrix} 1 & 0 \\ 0 & \sum_{j\ne i}
 -a_{ji}\zeta_j
\end{pmatrix} +g_i(z){\check \alpha_i} \;
+\rho_i(z)e_i.
\end{align}

This shows that in this gauge, the trace of the matrix of the
connection is constant with value $1-\sum_{j\ne i} a_{ji}\zeta_j$.

Thus, a $Z$-twisted Miura $G$-oper
gives rise to a collection of meromorphic Miura $\SL(2)$-opers
${\hat\nabla}_i(z)$ for $i=1,\ldots,r$. It should be noted that ${\hat\nabla}_i(z)$ has
regular singularities in the sense of Definition~\ref{d:regsing} if
and only if $\rho_i(z)$ is a polynomial. For example, this holds for
all $i$ if all $y_j(z), j=1,\dots,r$, are polynomials.  We will use
this observation below.

\subsection{$Z$-twisted Miura-Pl\"ucker opers}    \label{MP}

Recall that a $Z$-twisted Miura $G$-opers is a 
Miura $G$-oper whose underlying connection can
be written in the form \eqref{gaugeA}:
\begin{equation}    \label{gaugeA2}
\nabla(z)=v(z)(\partial_z+Z) v(z)^{-1}, \qquad v(z) \in B_+(z).
\end{equation}
We will now relax this condition by imposing a twistedness condition
only on the associated Miura $\GL(2)$-opers
$\nabla_i$ (or equivalently, the Miura $\SL(2)$-opers ${\hat
  \nabla}_i$). More precisely, we will require the existence of an
upper triangular gauge transformation $v(z)\in B_+(z)$ such that
\eqref{gaugeA2} holds upon restriction to $W_i$ for all $i$.

\begin{Def}    \label{ZtwMP}
  A $Z$-{\em twisted Miura-Pl\"ucker $G$-oper} is a meromorphic
  Miura $G$-oper on $\P^1$ with underlying connection $\nabla$
  satisfying the following condition: there exists $v(z) \in B_+(z)$
  such that for all $i=1,\ldots,r$, the Miura $\GL(2)$-opers
  $\nabla_i$ associated to $\nabla$ by formula \eqref{2flagformula} can be
  written in the form
\begin{equation}    \label{gaugeA3}
\nabla_i(z) = v(z)(\partial_z +Z) v(z)^{-1}|_{W_i} = v_i(z)(\partial_z +Z_i) v_i(z)^{-1},
\end{equation}
where $v_i(z) = v(z)|_{W_i}$ and $Z_i = Z|_{W_i}$.
\end{Def}

In other words, a Miura $G$-oper is a $Z$-twisted
Miura-Pl\"ucker $G$-oper precisely when there is a trivialization of
$\cF_{B_+}$ in which all of the associated connections $\nabla_i$ have
the constant matrix $Z_i\in\mathfrak{gl}(2)$.  It is a $Z$-twisted
Miura $G$-oper if $\nabla$ has the constant matrix $Z$ in this gauge.
Thus, every $Z$-twisted Miura $G$-oper is automatically a
$Z$-twisted Miura-Pl\"ucker $G$-oper, but the converse is not
necessarily true if $G \neq \SL(2)$.

Note, however, that it follows from the above definition that the
$H$-connection $\nabla^H$ associated to a $Z$-twisted
Miura-Pl\"ucker $G$-oper can be written in the same form
\eqref{AH2} as the $H$-connection associated to a $Z$-twisted
Miura $G$-oper.

\subsection{$H$-nondegeneracy}    \label{H nondeg}

We now introduce the notion of $H$-nondegeneracy, the first of our two nondegeneracy conditions for $Z$-twisted
Miura-Pl\"ucker opers. This condition actually applies to arbitrary Miura
opers with regular singularities. Recall from Theorem \ref{gen elt1} that the underlying connection can be represented in the
form \eqref{form of A}.

\begin{Def} \label{nondeg Cartan} A Miura $G$-oper $\nabla$ of the
  form \eqref{form of A} is called $H$-\emph{nondegene\-rate} if the
  corresponding $H$-connection $\nabla^H(z)$ can be written in the
  form \eqref{AH1}, with the rational functions $y_i(z)$ satisfying
  the following conditions: 
  \begin{enumerate}\item $y_i(z)$ has no multiple zeros
    or poles;
     \item for all $i$,  the roots of
      $\Lambda_i(z)$ are distinct from the the zeros and poles of
      $y_i(z)$; and
      \item if $i\ne j$ and $a_{ij}\ne 0$, then the zeros and poles of $y_i(z)$ and
  $y_j(z)$ are distinct from each other.
\end{enumerate}
\end{Def}


\subsection{Nondegenerate $Z$-twisted Miura
  $\SL(2)$-opers}    \label{nondeg sl2}

We now turn to the second nondegeneracy condition. This condition
applies to $Z$-twisted Miura-Pl\"ucker $G$-opers.   In this
subsection, we give the definition for $G=\SL(2)$. (Note that
$Z$-twisted Miura-Pl\"ucker $\SL(2)$-opers are the same as
$Z$-twisted Miura $\SL(2)$-opers.) In the next subsection, we will
give the definition for an arbitrary simple, simply connected complex
Lie group $G$.

Consider a Miura $\SL(2)$-oper given by the formula \eqref{form of A},
which for $\SL(2)$ becomes
$$
\nabla=\partial_z+{\check{\alpha}}g(z)+\Lambda(z)e  = \partial_z+\begin{pmatrix}
   g(z) & \Lambda(z) \\
   0 & -g(z)
  \end{pmatrix}.
$$
The corresponding Cartan
connection is given by
$$
\nabla^H(z)=\partial_z+g(z){\check{\alpha}} = y(z)^{\check{\alpha}}(\partial_z+Z^H)
y(z)^{-\check{\alpha}} = \partial_z+\begin{pmatrix}
   \zeta-y(z)^{-1}\partial_zy(z) & 0 \\
   0 & -\zeta+ y(z)^{-1}\partial_zy(z)
  \end{pmatrix},
$$
where $y(z)$ is a rational function. Let us assume that $A(z)$ is
$H$-nondegenerate, so that the zeros of $\Lambda(z)$ are distinct from the zeros and
poles of $y(z)$.

If we apply a gauge transformation by an element
$h(z)^{\check\alpha} \in H[z]$ to $\nabla$, we obtain a new oper
connection
\begin{equation}    \label{wtA}
\wt{\nabla}(z) = \partial_z+\wt{g}(z){\check{\alpha}} +\widetilde{\Lambda}(z)e,
\end{equation}
where
\begin{equation}    \label{wtg}
\wt{g}(z) = g(z) -h^{-1}(z) \partial_zh(z), \qquad \wt\Lambda(z) =
\Lambda(z) h(z)^2.
\end{equation}
It also has regular singularities, but for a
different polynomial $\wt{\Lambda}(z)$, and $\wt{\nabla}(z)$ may no longer
be $H$-nondegenerate.  However, it turns out there is an essentially
unique gauge transformation from $H[z]$ for which the resulting
$\wt{\nabla}(z)$ is $H$-nondegenerate  and $\wt{y}(z)$
is a polynomial.  This choice allows us to fix the polynomial
$\Lambda(z)$ determining the regular singularities of our
$\SL(2)$-oper.

\begin{Lem}    \label{nondegsl2}
\begin{enumerate}
\item There is an $H$-nondegenerate $\SL(2)$-oper $\wt{\nabla}(z)$ in
  the $H[z]$-gauge class of $\nabla$, say with
  $\wt{\nabla}^H(z)=\partial_z+\wt{g}(z){\check{\alpha}}$, for which the rational
  function $\wt{y}(z)$ is a polynomial. This oper is unique up to a
  scalar $a \in \C^\times$ that leaves $\wt{g}(z)$ unchanged, but
  multiplies $\wt{y}(z)$ and $\wt{\Lambda}(z)$ by $a$ and $a^2$
  respectively.
\item This $\SL(2)$-oper $\wt{\nabla}$ may also be characterized by
  the property that $\wt{\Lambda}(z)$ has maximal degree subject to
  the constraint that it is $H$-nondegenerate.
\end{enumerate}
\end{Lem}

\begin{proof} Write $y(z)=\frac{P_1(z)}{P_2(z)}$, where $P_1,P_2$ are
  relatively prime polynomials. For a nonzero polynomial $h(z) \in
  \C(z)^\times$, the gauge transformation of $\nabla$ by
  $h(z)^{\check{\alpha}}$ is given by formulas \eqref{wtA} and
  \eqref{wtg}.  In order for
  $\wt{y}(z)=h(z)\frac{P_1(z)}{P_2(z)}$ to be a polynomial, we need
  $h(z)$ to be divisible by $P_2(z)$. If, however, $\deg(h/P_2)>0$,
  then $\wt{y}(z)$ and $\wt{\Lambda}(z)$ would have a zero in common,
  so $\wt{A}(z)$ would not by $H$-nondegenerate.  Hence, we must have
  $h(z)=aP_2(z)$ for some $a\in\C^\times$. Thus, $h(z)$ is uniquely
  defined by multiplication by $a$, which leave $\wt{g}(z)$ unchanged,
  but multiplies $\wt{y}(z)$ and $\wt{\Lambda}(z)$ by $a$ and $a^2$
  respectively.

  For the second statement, note that if $h(z)$ is a polynomial for
  which the zeros of $h(z)^2\Lambda(z)$ are distinct from the
  zeros and poles of $h(z)\frac{P_1(z)}{P_2(z)}$, we must have
  $h|P_2$.  If $h(z)$ is not an associate of $P_2(z)$, we have
  $\deg(h)<\deg(P_2)$, so
  $\deg(h(z)^2\Lambda(z))<\deg(\wt{\Lambda})$.
\end{proof}

This motivates the following definition.

\begin{Def}    \label{ngsl2}
  A $Z$-twisted Miura $\SL(2)$-oper is called \emph{nondegenerate}
  if it is $H$-nondegenerate and the rational function $y(z)$
  appearing in formula \eqref{AH1} is a polynomial. 
\end{Def}

\subsection{Nondegenerate  Miura-Pl\"ucker
  $G$-opers}    \label{s:nondegenerate}

We now turn to the general case. Recall that to every
Miura-Pl\"ucker $G$-oper $\nabla$, we have associated a Miura
$\SL(2)$-oper 
${\hat \nabla}_i(z), i=1,\ldots,r$, given by formula \eqref{Bi}. (It is
obtained from the Miura $\GL(2)$-oper $\nabla_i = \nabla|_{W_i}$ using
formulas \eqref{cano} and \eqref{tilde calig}). It follows from the
definition that if $\nabla$ is $Z$-twisted  with $Z$ given by \eqref{Z},
then ${\hat \nabla}_i$ is $\zeta_i\check\alpha_i$-twisted.

\begin{Def}    \label{nondeg Miura}
  Suppose that the rank of $G$ is greater than 1. A $Z$-twisted
  Miura-Pl\"ucker $G$-oper $A(z)$ is called \emph{nondegenerate} if
  it is $H$-nondegenerate and
  each $\zeta_i\check\alpha_i$-twisted Miura
  $\SL(2)$-oper ${\hat \nabla}_i(z)$ is nondegenerate.
\end{Def}

It turns out that this simply means that in addition to $\nabla$ being
$H$-nondegenerate, each $y_i(z)$ from formula \eqref{AH1} is a
polynomial. 

\begin{Prop}    \label{nondeg1}
  Let $\nabla$ be a
  $Z$-twisted Miura-Pl\"ucker $G$-oper.
  The following statements are equivalent:
\begin{enumerate}
\item\label{nondegen1} $\nabla$ is nondegenerate.
  \item\label{nondegen2} $\nabla$ is $H$-nondegenerate, and each
    ${\hat \nabla}_i(z)$ has regular singularities, i.e. $\rho_i(z)$ given
    by formula \eqref{ri} is in $\C[z]$.
    \item\label{nondegen3} Each $y_i(z)$ from formula \eqref{AH1} may
      be chosen to be a monic
      polynomial, and these polynomials satisfy the conditions in
      Definition~\ref{nondeg Cartan}.   
\end{enumerate}
\end{Prop}

\begin{proof} To prove that \eqref{nondegen2} implies
  \eqref{nondegen3}, we need only show that if each $\rho_i(z)$ given
  by formula \eqref{ri} is in $\C[z]$, then the $y_i(z)$'s are
  polynomials.  Suppose $y_i(z)$ is not a polynomial, and choose $j\ne
  i$ such that $a_{ij} \neq 0$.  Then $-a_{ij}>0$, and so the
  denominator of $y_i(z)$ appears in the denominator of
  $\rho_j(z)$.  Moreover, since the poles of $y_i(z)$ are distinct
  from the zeros of $\Lambda_j(z)$ and the other $y_k(z)$'s, the poles
  of $y_i(z)$  give rise to poles of
  $\rho_j(z)$. But then ${\hat \nabla}_j(z)$ would not have regular
  singularities.

  Next, assume \eqref{nondegen3}.  By
  Definition \ref{nondeg Cartan}, $\nabla$ is $H$-nondegenerate.  Since all the $y_i(z)$'s are
  polynomials, the same is true for the $\rho_i(z)$'s.  (Here, we
  are using the fact that the off-diagonal elements of the Cartan
  matrix, $a_{ij}$ with $i\neq j$, are less than or equal to 0.)
  Since $\rho_i(z)$ is a product of polynomials whose roots are
  distinct from the roots of $y_i(z)$, we see that the Cartan
  connection associated to ${\hat \nabla}_i(z)$ is nondegenerate.

Finally, \eqref{nondegen2} is a trivial consequence of
\eqref{nondegen1}.
\end{proof}

If we apply a gauge transformation by an element $h(z)\in H[z]$ to
$\nabla$, we get a new $Z$-twisted Miura-Pl\"ucker $G$-oper.
However, the following proposition shows that it is only nondegenerate
if $h(z)\in H$.  As a consequence, the $\Lambda_k$'s of a
nondegenerate oper are determined up to scalar multiples. If we further impose the condition that each $y_i(z)$ is a {\em monic}
polynomial, then $h(z)=1$, and this fixes the $\Lambda_k$'s.

\begin{Prop} If $\nabla$ is a nondegenerate $Z$-twisted Miura-Pl\"ucker
  $G$-oper and $h(z)\in H[z]$, then $h(z)\nabla h(z)^{-1}$ is
  nondegenerate if and only if $h(z)$ is a constant element of $H$.
\end{Prop}

\begin{proof} Write $h(z)=\prod h_i(z)^{\check{\alpha}_i}$.  Gauge
  transformation of $\nabla$ by $h(z)$ induces a gauge transformation of
  $\nabla_i$ by $h_i(z)$.  Since $\nabla_i$ is nondegenerate,
  Lemma~\ref{nondegsl2} implies that the new Miura $\SL(2)$-oper
  is nondegenerate if and only $h_i\in\C^\times$.
\end{proof}

\section{$\SL(2)$-opers and the Bethe Ansatz equations}

Before exploring the relationship between Miura $G$-opers and the
Bethe Ansatz equations in general, we briefly describe what happens
for $G=\SL(2)$.  These results are immediate corollaries of the
results in the following sections.  However, in this case, one can
give simpler proofs; see \cite{KSZ} for the details.

Let $Z^H=\diag(\zeta,-\zeta)$.  A nondegenerate $Z^H$-twisted Miura
$\SL(2)$-oper can be represented in matrix form as
\begin{equation*}
\nabla(z)=\partial_z+(\zeta-y(z)^{-1}\partial_zy(z))\check{\alpha}+\Lambda(z)e=\begin{pmatrix}
   \zeta-y(z)^{-1}\partial_zy(z) & \Lambda(z) \\
   0 & -\zeta+ y(z)^{-1}\partial_zy(z)
  \end{pmatrix},
\end{equation*}
where the polynomials $y(z)$ and $\Lambda(z)$ have no roots in common
and $y(z)$ is monic with no multiple roots.  This connection is gauge
equivalent to $\partial_z+\zeta\check \alpha+\Lambda(z)e$ via a gauge
transformation by a matrix of the form
\begin{equation*} v(z)=y(z)^{\check\alpha}e^{\frac{q_-(z)}{q_+(z)}e},
\end{equation*}
where $q_-(z),q_+(z)$ are relatively prime polynomials with $q_+(z)$
monic.

One can now show that $y(z)=q_+(z)$ and the polynomials $q_+(z)$ and
$q_-(z)$ satisfy the following differential equation involving their
Wronskian:
\begin{equation*}
q_{+}(z)\partial_z q_-(z)-q_{-}(z)\partial_z q_+(z)+2\zeta q_+(z)q_-(z)=\Lambda(z)
\end{equation*}
This is the $\SL(2)$-version of a system of equations called the
$qq$-system.   In fact, there is a bijection between nondegenerate
$Z^H$-twisted Miura opers together with a choice of the matrix $v(z)$ and nondegenerate polynomial solutions of
the $qq$-system; here, a polynomial solution of the $qq$-system is called
nondegenerate if $q_{+}(z)$ is monic with no multiple roots and has no
roots in common with $\Lambda(z)$.

Nondegenerate solutions lead to solutions of the Bethe
Ansatz equation for the inhomogeneous Gaudin model.  Indeed, let
$\Lambda(z)=\prod^N_{k=1}(z-z_k)^{\ell_i}$ and
$q_+(z)=\prod^n_{i=1}(z-w_i)$ with $w_i\ne w_{j}$ if $i\ne j$ and
$w_i\ne z_k$.  One can then show that
\begin{equation}
2\zeta+\sum^N_{k=1}  \frac{\ell_k}{w_i-z_k}-\sum^n_{k=1} \frac{2}{w_i-w_k}=0, \quad k=1,\dots, r.
\end{equation}
In fact, there is a one-to-one correspondence between $Z^H$-twisted
Miura opers and solutions of the Bethe Ansatz equation.

\section{Miura-Pl\"ucker opers, Wronskian relations, and the Bethe Ansatz
  equations for the Gaudin model}    \label{Sec:QQsystem}

We now return to the general situation, with $G$ an arbitrary simple, simply connected
complex Lie group.  We show that a $Z$-twisted Miura-Pl\"ucker
$G$-oper is also $Z^H$-twisted.  We then establish a one-to-one
correspondence between the set of nondegenerate $Z^H$-twisted
Miura-Pl\"ucker $G$-opers and the set of 
solutions of a system of Bethe Ansatz equations associated to
$G$. A key
element of the construction is an intermediate object between these
two sets: solutions to a system of nonlinear differential equations
called the \emph{$qq$-system}, which imposes relations on certain
Wronskians indexed by the simple roots.

\subsection{Reduction to the semisimple case}

Let $\nabla $ be a $Z$-twisted Miura-Pl\"ucker oper for $Z\in\fb_+$.
As in \eqref{Z}, we write $Z=Z^H+\sum^r_{i=1}c_ie_i+n_+$ with $Z^H=
\sum^r_{i=1}\zeta_i\check\alpha_i\in\mathfrak{h}$ and $n_+\in[\mathfrak{n}_+,\mathfrak{n}_+]$. 

We now show that a $Z$-twisted  Miura-Pl\"ucker oper is also
$Z^H$-twisted. 

\begin{Prop}
i) There exist an element $u(z)\in N_+(z)$ so that $u(z)(\partial_z+Z)u(z)^{-1}= 
\partial_z+Z^H+\tilde{n}_+(z)$, where 
 $\tilde{n}_+(z)\in [\mathfrak{n}_+,\mathfrak{n}_+](z)$.\\
ii) Any $Z$-twisted Miura-Pl\"ucker oper is $Z^H$-twisted.
\end{Prop}
\begin{proof}

To prove the first statement, we will construct $u(z)$ as a
product of $r$ elements corresponding to the simple
roots. Assume that $\langle \alpha_i, Z^H\rangle\neq 0$, and set $u_i(z)=\exp\Big(-\frac{c_i }{\langle  \alpha_i,Z^H\rangle}e_i\Big)$.
We obtain 
\begin{equation}\label{elimi}
u_i(z)(\partial_z+Z)u_i(z)^{-1}=Z^H+\sum^r_{j=1, i\neq j}c_je_j+\dots,
\end{equation}
where the dots stand for terms in
$[\mathfrak{n}_+,\mathfrak{n}_+](z)$. Similarly, if $\langle 
\alpha_i,Z^H\rangle= 0$, set 
$u_i(z)=\exp(zc_ie_i)$, which again leads to \eqref{elimi}.
Then $u(z)=\prod^r_{i=1}u_i(z)$, where the order of the $u_i(z)$'s
does not matter, satisfies the desired conditions.

Recall that we have $v(z)\in B_+(z)$ such that $\nabla_i(z) = v(z)(\partial_z +Z)
v(z)^{-1}|_{W_i}$ for all $i$.   Set $v^u(z)=v(z)u(z)^{-1}\in B_+(z)$,
with $u(z)$ as in the first part.   It follows that
\begin{equation}
\begin{aligned}  
\nabla_i(z) &= v(z)(\partial_z +Z) v(z)^{-1}|_{W_i} \\
&=v(z)u(z)^{-1}(\partial_z +Z^H) u(z)v(z)^{-1}|_{W_i}=v^u_i(z)(\partial_z +Z^H_i) v^u_i(z)^{-1}.
\end{aligned}
\end{equation}
where $v^u_i(z) = v(z)u^{-1}(z)|_{W_i}$ and $Z^H_i = Z^H|_{W_i}$. Thus any $Z$-twisted Miura-Pl\"ucker oper is $Z^H$-twisted.

\end{proof}

For the rest of the paper, we will restrict attention to opers with a
semisimple twist.  However, we will retain the notation $Z^H$ for
clarity.

\subsection{Twisted Miura-Pl\"ucker data and $qq$-systems}

We now introduce a nonlinear system of differential equations
depending on the polynomials $\Lambda_1(z),\dots, \Lambda_r(z)$ and
the semisimple element $Z^H$.  As we will see, it may be viewed as a
functional realization of the Bethe Ansatz equations.

Recall that the Wronskian of two rational functions $q_+(z)$ and
$q_-(z)$ is given by
\begin{equation*} W(q_+,q_-)(z)=q_+(z)\partial_z
  q_-(z)-q_-(z)\partial_z q_+(z).
\end{equation*}

\begin{Def}  The \emph{$qq$-system} associated to $\mathfrak{g}$, the
  semisimple element $Z^H\in\mathfrak{h}$,
  and the collection of monic polynomials $\Lambda_1(z),\dots,\Lambda_r(z)$ is
  the system of equations
\begin{equation}    \label{qq}
{W(q^i_+,q^i_-)(z)}+\langle   \alpha_i,Z^H\rangle
{q^i_+(z)q^i_-(z)}=\Lambda_i(z)\prod_{j\neq i}\Big [ q_+^j(z)\Big
]^{-a_{ji}}
\end{equation}
\end{Def}
for $i=1,\dots,r$.

A polynomial solution $\{ q^i_+(z),q^i_-(z) \}_{i=1,\ldots,r}$ of
\eqref{qq} is called {\em nondegenerate} if each $q^i_+(z)$ is monic
and the $q^i_+(z)$'s satisfy the conditions in
Definition~\ref{nondeg Cartan}.  Note that nondegeneracy only depends
on the $q^i_+(z)$'s.

It is an immediate consequence of the definition that for
nondegenerate polynomial solutions, $q^i_+(z)$ and
$q^i_-(z)$ are relatively prime.  Indeed, if $w$ is a
common root of $q^i_+(z)$ and
$q^i_-(z)$, then it is a root of the left-hand side of the $i$th
$qq$-equation.  It follows that $w$ is also a root of some factor on
the right-hand side, which contradicts nondegeneracy.

\begin{Rem}
This system of equations (\ref{qq}) has also been considered in \cite{mukhvarmiura} in the context of differential operators corresponding to Miura opers for $Z=0$.
\end{Rem}

\begin{Rem}
If $\fg$ is not simply-laced, let $\tilde{\fg}$ be the associated
simply-laced Lie algebra, i.e., the Lie algebra whose Dynkin diagram
has the multiple bond replaced by a simple bond.  (We will
systematically use tilde superscripts to denote objects associated to
this new Lie algebra.)  We suppose further
that $\fg$ has a unique short simple root, hence is of type $B_n$ or
$G_2$.    In this case, we show that a solution to the $qq$-system for
$\fg$ gives rise to a solution to the $qq$-system for $\tilde{\fg}$.

Let $\{q^i_+(z),
q^i_-(z)\}$ be a solution to the $qq$-system for $\fg$ for fixed
$Z^H$ and $\Lambda_i$'s.  We let $k$ and $\ell$ be the indices
of the simple roots connected by the multiple bond, with $k$
corresponding to the short simple root.  Note that the Cartan matrices
of $\fg$ and $\tilde{fg}$ only differs in the $k,\ell$ entry.

Fix a semisimple element $\tilde{Z}^{\tilde{H}}\in\tilde{\fh}$ by the
equations $\langle   \tilde{\alpha}_i,
\tilde{Z}^{\tilde{H}}\rangle=(1+\delta_{ik}(-a_{k\ell}-1))\langle   \alpha_i,Z^H\rangle$.  Define
polynomials $\tq^i_{\pm}(z)$ and $\tilde{\Lambda}_i(z)$
by \begin{equation*}\tq^i_{\pm}(z)=\begin{cases}
    (q^k_{\pm}(z)) ^{-a_{k\ell}} & i=k,\\
    q^i_{\pm}(z) & \text{otherwise}, 
  \end{cases}  \qquad
\tilde{\Lambda}_i(z)=\begin{cases}
   -a_{k\ell}(q^k_{+}(z)q^k_-(z)) ^{-a_{k\ell}-1}{\Lambda}_k(z) &
   i=k,\\
   (q^k_{+}(z))  ^{-a_{k\ell}-1}{\Lambda}_\ell(z) & i=\ell,\\
     {\Lambda}_i(z) & \text{otherwise}.
  \end{cases}
\end{equation*}  (Note that $\tilde{\Lambda}_k$ is no longer monic.)

It is now easy to check that the $\tq^i_{\pm}(z)$'s satisfy the
$qq$-system for $\tilde{\fg}$ given by

\begin{equation*}   
{W(\tq^i_+,\tq^i_-)(z)}+\langle   \tilde{\alpha}_i,\tilde{Z}^{\tilde{H}}\rangle
{\tq^i_+(z)\tq^i_-(z)}=\tilde{\Lambda}_i(z)\prod_{j\neq i}\Big [ \tq_+^j(z)\Big
]^{-\tilde{a}_{ji}}.
\end{equation*}
The $k$th equation is just the original $k$th equation multiplied by
$-a_{k\ell}(q^k_{+}(z)q^k_-(z)) ^{-a_{k\ell}-1}$.  The left-hand sides
of the new and old $\ell$th equations coincide, and the additional
factor in $\tilde\Lambda_\ell(z)$ ensures that the same holds for the
right-hand sides.  Finally, suppose $i\ne k,\ell$.  Since $i$ is not
connected to $k$, $\tilde{a}_{ki}=a_{ki}=0$, $q^i_+$ and $\tq^i_+$ do
  not appear on the right-hand side of the $i$th equation, so the new
  and old equations are identical.  Note that this is where the
  construction fails if types $C_n$ and $F_4$.

  We remark that this construction always leads to degenerate
  solutions of the $qq$-system for $\tilde{\fg}$.  

\end{Rem}

\begin{Rem}
The $q$-deformed version of the system (\ref{qq}) is known as a
$QQ$-system \cite{Frenkel:ac}.  It plays a similar role in the study
of the Bethe Ansatz equations for the XXZ model. It also arises in the
ODE/IM correspondence \cite{Masoero_2016,Masoero_2016_SL}, in the
representation theory of quantum groups \cite{Frenkel:ac}, and in enumerative geometry \cite{Koroteev:2017aa,KSZ,KZtoroidal,KZ3d}.  
\end{Rem}

In order to describe the relationship between solutions of the
$qq$-system and Miura-Pl\"ucker opers, we need the notion of a
\emph{$Z^H$-twisted Miura-Pl\"ucker datum}.  Recall that if $\nabla$
is a $Z^H$-twisted Miura-Pl\"ucker oper, then by Theorem \ref{gen elt1}, it can be
  written in the form \eqref{form of A}:
\begin{equation}    \label{form of A1}
\nabla=\partial_z+\sum^r_{i=1}
g_i(z)\check{\alpha}_i+\sum^r_{i=1}{\Lambda_i(z)e_i}, \qquad
g_i(z) \in \C(z)^\times.
\end{equation}
Moreover, there exists $v(z) \in B_+(z)$ such that for all $i=1,\ldots,r$,
the Miura $\GL(2)$-opers $\nabla_i$ associated to $\nabla$ can be written in the form \eqref{gaugeA3}:
\begin{equation}    \label{gaugeA4}
\nabla_i = v_i(z)(\partial_z+Z^H_i)v_i(z)^{-1}, \qquad i=1,\ldots,r,
\end{equation}
where $v_i(z) = v(z)|_{W_i}$ and $Z^H_i = Z^H|_{W_i}$.

The element $v(z)$ is not uniquely determined by the Miura-Pl\"ucker
oper.  First, note that the subgroup $[N_+(z),N_+(z)]$ acts trivially on the
representations $W_i$.  Next, it is obvious from \eqref{gaugeA4} that the constant
maximal torus $H$ fixes $\partial_z+Z^H_i$.  It follows that any element of the coset $v(z) H[N_+(z),N_+(z)]$
also satisfies \eqref{gaugeA4}. We call such a coset a
\emph{framing} of the Miura-Pl\"ucker oper.

\begin{Def}  A \emph{$Z^H$-twisted Miura-Pl\"ucker datum} is a pair
  $(\nabla, v(z) H[N_+(z),N_+(z)])$ consisting of a $Z^H$-twisted
  Miura-Pl\"ucker oper together with a framing.  The datum is called
  nondegenerate if the underlying Miura-Pl\"ucker oper is nondegenerate.
\end{Def}

\begin{Thm}    \label{inj}
  There is a one-to-one correspondence between the set of
  nondegenerate $Z^H$-twisted Miura-Pl\"ucker data and the set
  of nondegenerate polynomial solutions of the $qq$-system \eqref{qq}.
\end{Thm}

\begin{proof}
   Let $(\nabla,  v(z) H[N_+(z),N_+(z)])$ be a nondegenerate $Z^H$-twisted Miura-Pl\"ucker
  datum. We will fix the representative of the framing coset by setting
\begin{equation}    \label{vdots}
v(z) = \prod_{i=1}^r y_i(z)^{\check\alpha_i} \prod_{i=1}^r
e^{-\frac{q^i_{-}(z)}{q^i_{+}(z)} e_i},
\end{equation}
where $q^i_+(z), q^i_-(z)$ are relatively prime polynomials with
$q^i_+(z)$ monic for each $i=1$ and each
$y_i(z)$ is a monic polynomial.

We now show that the $q^i_+(z), q^i_-(z)$'s give a nondegenerate
solution to the $qq$-system and in fact,
\begin{equation}    \label{yiQi}
y_i(z)=q_+^i(z), \qquad i=1,\ldots,r.
\end{equation}

We first compute the matrix of $v(z)$ and $Z^H$ acting on the
two-dimensional subspace $W_i$ introduced in Section \ref{rank2}. A
short calculation shows that
\begin{equation}    \label{vzz}
v(z)|_{W^i}=
\begin{pmatrix}
  y_i(z) & 0\\
  0& y_i^{-1}(z)\prod_{j\neq i} y_j^{-a_{ji}}(z)
 \end{pmatrix}
 \begin{pmatrix}
1 & - \frac{q^i_{-}(z)}{q^i_{+}(z)}\\
 0& 1
 \end{pmatrix}
\end{equation}
and 
\begin{equation}
Z^H|_{W_i}=\begin{pmatrix}
\zeta_i & 0\\
  0& -\zeta_i-\sum_{j\neq i}{a_{ji}}\zeta_j
   \end{pmatrix}.
\end{equation}

We now apply \eqref{2flagformula} and \eqref{gaugeA4} to relate the
$y_i(z)$'s and $q^i_{\pm}(z)$'s. First, comparing the diagonal entries
on both sides of \eqref{gaugeA4} gives formula \eqref{giyi}:
\begin{equation}    \label{giz}
g_i(z)=\zeta_i-{y_i^{-1}(z)}\partial_z{y_i(z)}.
\end{equation}
Next, by comparing the upper triangular entries on both sides of
\eqref{gaugeA4}, we obtain
\begin{equation}    \label{Lambdai}
\Bigg[\partial_z\Bigg(\frac{q^i_-(z)}{q^i_+(z)}\Bigg)+\Big(\sum_{j}a_{ji}\zeta_j\Big)\frac{q^i_-(z)}{q^i_+(z)}\Bigg]\Big[{y_i(z)}\Big]^2=\Lambda_i(z)\prod_{j\neq i}y_j(z)^{-a_{ji}}.
\end{equation}
Multiplying through by $q^i_+(z)^2$ gives

\begin{equation}    \label{Lambdai2}
\Big[W(q^i_+(z), q^i_-(z))+\Big(\sum_{j}a_{ji}\zeta_j\Big)q^i_-(z)q^i_+(z)\Big]\Big[{y_i(z)}\Big]^2=\Big[{q^i_+(z)}\Big]^2\Lambda_i(z)\prod_{j\neq i}y_j(z)^{-a_{ji}}.
\end{equation}

The nondegeneracy conditions for our oper imply that
$y_i(z)|q^i_+(z)$.  Write $q^i_+(z)=y_i(z)p(z)$.  We will show that
$p(z)$ has degree $0$.  Suppose that $p(z)$ has a root $c$ with
multiplicity $m\ge 1$.  Note that $c$ is a root of $q^i_+$ of
multiplicity either $m$ or $m+1$, depending on whether $c$ is a
(necessarily simple) root of $y_i(z)$.

Now, rewrite the previous equation as
\begin{equation} \label{Lambdai3} q^i_-(z)\partial_z
  q^i_+(z)=q^i_+(z)\partial_z
  q^i_-(z)+\Big(\sum_{j}a_{ji}\zeta_j\Big)q^i_-(z)q^i_+(z)-p(z)^2\Lambda_i(z)\prod_{j\neq
    i}y_j(z)^{-a_{ji}}.
\end{equation}

Suppose that $c$ is not a root of $y_i(z)$.  Then $c$ is a root of the
left-hand side of \eqref{Lambdai3} with multiplicity $m-1$.  Since $c$ is a
zero of the 
three terms on the right-hand side have multiplicities $\ge m$, $m$, and
$2m$ respectively, we have a contradiction.  On the other hand, if $c$
is a root of $y_i(z)$, then $c$ is a root of the left-hand side with
multiplicity $m$ while it is a root of the three terms on the right-hand
side with multiplicities $\ge m+1$, $m+1$, and $2m$.  Again, we have a
contradiction, so $p(z)$ is a constant.  Since $q^i_+(z)$ and $y_i(z)$ are
monic, $p(z)=1$ and $q^i_+(z)=y_i(z)$.

Dividing out by $y_i(z)^2$ in \eqref{Lambdai2}, we see that the
polynomials $q^i_+(z), q^i_-(z)$, $i=1,\ldots,r$, satisfy the system
of equations \eqref{qq} and are nondegenerate.  Thus, we obtain a map
from the set of nondegenerate Miura $G$-opers to the set of
nondegenerate solutions of \eqref{qq}.

To show that this map is a bijection, we construct its
inverse. Suppose that we are given a nondegenerate solution $\{
q^i_+(z), q^i_-(z) \}_{i=1,\ldots,r}$ of the system \eqref{qq}. 
We then define $\nabla$ by formula
\eqref{form of A1}, where we set
$$
g_i(z)=\zeta_i -{q^i_+}(z)^{-1}\partial_z{q^i_+(z)},
$$
 i.e.
\begin{equation} \label{key}
\nabla =\partial_z + \sum^r_{i=1}\Big [\zeta_i -{q^i_+}(z)^{-1}\partial_z{q^i_+(z)}\Big]\check{\alpha}_i+\sum^r_{i=1}\Lambda_i(z)e_i.
\end{equation}
We also set
\begin{equation}    \label{prodi}
v(z) = \prod_{i=1}^r q^i_+(z)^{\check\alpha_i} \prod_{j=1}^r
e^{-\frac{q^j_{-}(z)}{q^j_{+}(z)} e_i}.
\end{equation}
Note that this means that we are setting $y_i(z)=q^i_+(z)$ for all
$i$.  Equations \eqref{gaugeA4} are now satisfied for all $i$.
Indeed, the Wronskian equations imply that the off-diagonal part of
\eqref{gaugeA4} holds while the diagonal part is automatic.  Moreover,
the nondegeneracy conditions on $\nabla$ are satisfied by Proposition
\ref{nondeg1}. Therefore, $(\nabla,  v(z) H[N_+(z),N_+(z)])$  defines a nondegenerate
$Z^H$-twisted Miura-Pl\"ucker $G$-oper. This completes the proof.
\end{proof}

\begin{Rem}  The inverse map is defined even for degenerate
  solutions of the $qq$-system.  Thus, a polynomial solution
  of the $qq$-system gives rise to a  $Z^H$-twisted Miura-Pl\"ucker datum without the assumption of nondegeneracy.
\end{Rem}

\begin{Cor}\label{MPsurj}  There is a surjective map from the set of nondegenerate
  polynomial solutions of the $qq$-system \eqref{qq} to the set of
  nondegenerate $Z^H$-twisted Miura-Pl\"ucker opers whose fibers
  consist of all solutions with fixed $q^i_+(z)$'s.
\end{Cor}

\begin{proof}  
In the correspondence of the theorem, the  Miura-Pl\"ucker oper is defined
entirely in terms of the $q^i_+(z)$'s.  The desired map is the
composition of the inverse map with the map that forgets the framing.
\end{proof}

In the next section, we will describe the fibers of this map
explicitly.

\subsection{The $qq$-system and the Bethe Ansatz equations}

We now derive the equations determining the zeros of a nondegenerate
polynomial solution $\{q^i_+(z),q^i_-(z)\}_{i=1,\dots,r}$ of the
$qq$-system.  These equations are precisely the Bethe Ansatz equations
for the inhomogeneous Gaudin model that were introduced
in~\cite{Feigin:2006xs,Rybnikov:2010}.

We begin by reformulating the $qq$-system.  Multiplying both sides of
\eqref{qq} by $q^i_+(z)^{-2}e^{\langle \alpha_i,Z^H\rangle z}$ and
recalling that $a_{ii}=2$, we see that the $qq$-system is equivalent
to

\begin{equation}  \label{qqred}
\partial_z\Bigg[e^{\langle\alpha_i,Z^H\rangle z}\Bigg(\frac{q^i_-(z)}{q^i_+(z)}\Bigg)\Bigg]=\Lambda_i(z)\left(\prod_{j}q_+^j(z)^{-a_{ji}}\right)e^{\langle \alpha_i,Z^H\rangle z}, \quad i=1, \dots, r.
\end{equation}

Let $\{w^i_\ell\}$ be the roots of $q^i_+(z)$. To derive the Bethe Ansatz
equations, recall that a meromorphic function
$f(z)$ with a double pole at $w$ has residue $0$ if and only if
$\partial_z\log(f(z)(z-w)^2)|_{z=w}=0$.  By nondegeneracy, we can apply this remark to the right-hand side of
\eqref{qqred} at $w^i_\ell$, thereby obtaining the system of equations

\begin{equation}\label{bethe}\begin{gathered}
\langle\alpha_i,Z^H\rangle+\partial_z\log\Big[\Lambda_i(z)\prod_{j}q_+^j(z)^{-a_{ji}}(z-w^i_k)^2\Big]\Bigg|_{z=w^i_\ell}=0,\\ 
 i=1,\dots, r; \quad  \ell=1, \dots, \deg(q^i_+(z)).
\end{gathered}
\end{equation}

These equations can be recast in a more familiar form by
computing the logarithmic derivatives explicitly.  Recall from \eqref{lambdaroots}
that the roots of the $\Lambda_j(z)$'s are denoted by $z^i_1,\dots,z^i_{N_i}$
and the multiplicity of the root $z_k$ in the $\Lambda_j(z)$'s is
determined by the dominant integral coweight $\check{\lambda}_k$.  A simple
computation now gives the Bethe Ansatz equations
\begin{equation}\label{BAEexplicit} \begin{gathered}
\langle \alpha_i,Z^H\rangle+\sum^{N_i}_{j=1}\frac{\langle \alpha_i,\check{\lambda}_j\rangle}{w_\ell^i-z^i_ j}-\sum_{(j,s) \neq (i,\ell)} \frac{a_{ji}}{w_\ell^i-w_s^j}=0,\\ 
i=1,\dots, r, \quad  \ell=1, \dots, \deg(q^i_+(z)).
\end{gathered}
\end{equation}

\begin{Rem} These are the Bethe Ansatz equations corresponding to the
  representation $\otimes_{j=1}^N V_{\check{\lambda}_j}$ and the coweight
  $\check{\mu}=\sum \check{\lambda}_j -\sum\deg q^i_+(z)
  \check{\alpha}_{i}$.  The $\check{\lambda}_i$'s are dominant, but we are not
assuming that $\check{\mu}$ is dominant. 
\end{Rem}

Next, we show that the map from nondegenerate polynomial solutions of
the $qq$-equations to solutions of the Bethe Ansatz equations is
surjective; moreover, the fibers are affine spaces of dimension equal
to the number of simple roots which kill $Z^H$.

We start by considering some properties of the rational functions
$\phi_i(z)=q^i_-(z)/ q^i_+(z)$.  First, we get an equivalent form of
the $i$th $qq$-equations by dividing \eqref{qq} by $q^i_+(z)^2$:
\begin{equation}\label{phieq} \partial_z \phi_i(z)+\langle
  \alpha_i,Z^H\rangle\phi_i(z)=\Lambda_i(z)\left(\prod_{j}q_+^j(z)^{-a_{ji}}\right).
\end{equation}
For convenience, we set $\xi_i=\langle
\alpha_i,Z^H\rangle$.

Since $e^{\xi_i
  z}\Lambda_i(z)\left(\prod_{j}q_+^j(z)^{-a_{ji}}\right)$ has a double root at $w^i_k$ and residue $0$, we obtain the
 partial fraction decomposition
\begin{equation}\Lambda_i(z)\left(\prod_{j}q_+^j(z)^{-a_{ji}}\right)=p_i(z)+\sum b^i_k
  \left(\frac{1}{(z-w^i_k)^2}-\frac{\xi_i}{z-w^i_k}\right),
\end{equation}
where $p_i(z)$ is a polynomial.
If we write \begin{equation}\label{phipf} \phi_i(z)=h_i(z)+\sum
  \frac{c^i_k}{z-w^i_k},
\end{equation}
with
$h_i(z)$ a polynomial, then \eqref{phieq} can be expressed in terms of
partial fraction decompositions as
\begin{equation}\partial_z h_i(z)+\xi_i h_i(z) -\sum
  \frac{c^i_k}{(z-w^i_k)^2}+\sum \frac{\xi_i c^i_k}{z-w^i_k}=p_i(z)+\sum b^i_k
  \left(\frac{1}{(z-w^i_k)^2}-\frac{\xi_i}{z-w^i_k}\right).
\end{equation}
In other words,
\begin{equation}\label{inverseconditions} c^i_k=-b^i_k \qquad \text{and}\qquad  \partial_z
  h_i(z)+\xi_i h_i(z)=p_i(z).
\end{equation}

We will use these conditions to define a polynomial solution of the
$qq$-systems associated to a solution of the Bethe Ansatz equations.
Fix such a solution, i.e., a collection of $w^i_\ell$'s satisfying
\ref{BAEexplicit}.  Notice that for this solution to make sense,
$w^i_\ell$ is not a root of $\Lambda_i$ and if $a_{ji}\ne 0$ and
$(i,\ell)\ne (j,s)$, then 
$w^i_\ell\ne w^j_s$.  Set $q^i_+(z)=\prod_\ell (z-w^i_\ell)$.   We
must show that there exist polynomials $q^i_-(z)$ which extend the
$q^i_+(z)$'s to a solution of \eqref{qq}; this solution will
automatically be nondegenerate.

In order to define $q^i_-(z)$, we will construct a rational function $\phi_i(z)$
whose poles are precisely the roots of $q^i_+(z)$ and set
$q^i_-(z)=\phi_i(z)q^i_+(z)$.  We define $\phi_i(z)$ via the partial
fraction decomposition \eqref{phipf}, so that the 
$qq$-equations are satisfied if and only if \ref{inverseconditions}
holds.  Thus, after setting $c^i_k=-b^i_k$, we just need $h_i(z)$ to
be a polynomial solution of the differential equation $\partial_z
  h_i(z)+\xi_i h_i(z)=p_i(z).$  If $\xi_i=0$, then $h_i(z)$ can be any
  indefinite integral of $p_i(z)$.  If $\xi_i\ne 0$, then there is a
  unique indefinite integral of $e^{\xi_i z}p_i(z)$ such that
  $h_i(z)=e^{-\xi_i z}\int^z e^{\xi_i x}p_i(x)\,dx$ is a polynomial.

We have thus proved the following theorem.

\begin{Thm}\label{BAE}
\begin{enumerate}\item
If $\langle \alpha_i,Z^H\rangle\ne 0$ for all $i$ (for example, if
$Z^H$ is regular semisimple), then there is a bijection between the solutions of the
Bethe Ansatz equations \eqref{bethe} and the nondegenerate polynomial
solutions of the $qq$-system \eqref{qq}. \\
\item  If $\langle \alpha_{l},Z^H\rangle=0$, 
for $l=i_1,\dots, i_k$ and is nonzero otherwise, then
$\{q^i_+(z)\}_{i=1,\dots, r}$ and 
$\{q^i_-(z)\}_{i\neq {i_1,\dots, i_k}}$, are uniquely determined by
the Bethe Ansatz equations, but each $\{q^{i_j}_-(z)\}$ for $j=1,\dots
k$ is only determined up to an arbitrary transformation
$q^{i_j}_-(z)\to q^{i_j}_-(z)+c_jq^{i_j}_+(z)$, where $c_j\in
\mathbb{C}$.
\end{enumerate}
\end{Thm}

\begin{Rem} The map $\{q^i_+(z), q^i_-(z)\}\mapsto \{q^i_+(z)\}$
  taking polynomial solutions of the $qq$-system to the ``positive
  part'' has fibers which are affine spaces of the dimension given in
  the theorem, even when the solutions are degenerate.  Indeed, choose $q^1_+(z),\dots, q^r_+(z)$ for which
there exists a (not necessarily nondegenerate) polynomial solution of the
$qq$-system. The possible $q^i_-(z)$'s are determined by integrating
\eqref{qqred}: 
\begin{equation} 
q^i_-(z)=q^i_+(z)e^{-\langle \alpha_i,Z^H\rangle z}\int^z e^{\langle \alpha_i,Z^H\rangle x}\Lambda_i(x)\prod_{j}q_+^j(x)^{-a_{ji}}dx,
\end{equation}
Here, we must choose the integration constant so that $q^i_-(z)$
is a polynomial.  By hypothesis, there exists at least one such constant.

If $\langle \alpha_i,Z^H\rangle\ne 0$ for all $i$ (for example, if
$Z^H$ is regular semisimple), then it is clear that only one
integration constant is possible, so the $q^i_-(z)$'s are uniquely
determined.  However, if $\langle \alpha_i,Z^H\rangle= 0$, then
$q^i_-(z)$ is only determined up to adding a constant multiple of
$q^i_+(z)$:
\begin{equation} 
q^i_-(z)=q^i_+(z)\Bigg[ c_i+\int^z
\Lambda_i(x)\prod_{j}q_+^j(x)^{-a_{ji}}dx\Bigg],
\end{equation}
where $c_i\in\C$ is arbitrary.
\end{Rem}

\begin{Rem} \label{degrees} The previous remark shows that the degrees of the $q^i_-$'s are
  essentially determined by the degrees of the $q^i_+$'s and the
  $\Lambda_i$'s.  If $\langle \alpha_{i},Z^H\rangle\ne 0$, then it is
  obvious from the $i$th $qq$-equation that $\deg
  q^i_-=\deg\Lambda_i-\deg q^i_+-\sum_{j\ne i}a_{ji}\deg q^j_+$.  On
  the other hand, if $\langle \alpha_{i},Z^H\rangle= 0$, then it
  follows from the theorem that there is a solution with $\deg
  q^i_-\ne \deg q^i_+$.  In this case, $\deg W(q^i_+, q^i_-)=\deg
  q^i_{+} +\deg q^i_{-} -1$, so $\deg q^i_-=1+\deg\Lambda_i-\deg
  q^i_+-\sum_{j\ne i}a_{ji}\deg q^j_+$.  If this degree is greater
  than $\deg q^i_+$, then every possible  $q^i_-$ has this degree.  If
  it is less than $\deg q^i_+$, then every other possible $q^i_-$ has
  degree equal to $\deg q^i_+$.
\end{Rem}


An immediate consequence of this theorem is the algebraicity of the set of $q^i_+$'s
giving rise to nondegenerate solutions of the $qq$-system.  More
precisely, fix nonnegative integers $d_1,\dots,d_r$.  Let
$\mathcal{Q}_{d_1,\dots,d_r}$ be the set of monic polynomials  $p_1,\dots p_r$ such that
  there exists a nondegenerate polynomial solution of the
  $qq$-equations (for the given $Z^H$ and $\Lambda_i$'s) satisfying
  $q^i_+=p_i$ and $\deg p_i=d_i$ for all $i$.

\begin{Cor}  The set $\mathcal{Q}_{d_1,\dots,d_r}$ is an affine variety.
\end{Cor}

This theorem states that there is a surjection from nondegenerate polynomial
solutions of the $qq$-system and solutions of the Bethe Ansatz
equation whose fibers
  consist of all solutions with fixed $q^i_+(z)$'s.  Combining this
  with Corollary~\ref{MPsurj} gives the following result:
\begin{Thm}\label{MPbethe}  There is a one-to-one correspondence
  between nondegenerate $Z^H$-twisted Miura-Pl\"ucker opers and
  solutions of the Bethe Ansatz equations \eqref{bethe}.
\end{Thm}

\begin{Rem} It is easy to show that the Laurent expansion at $\infty$
  of a $Z$-twisted Miura-Pl\"ucker oper is given by formula (6.9) of
  \cite{Feigin:2006xs}, which is the principal part at $\infty$  of a certain oper with
 an irregular singularity.
\end{Rem}

\subsection{Regularity of the connection at the $\{w^i_\ell\}$'s}

The expression \eqref{key} for a nondegenerate Miura-Pl\"ucker oper
appears to have singularities at the roots of the $q^i_+$'s.  However,
the connection is in fact regular there.  To show this, it will be
convenient to describe the Bethe Ansatz equations in terms of the
Cartan connection $\nabla^H=\partial_z+A^H(z)$, with $A^H(z)$ defined
in \eqref{AH}:

  \begin{equation}\label{Bethecartan}
    \begin{gathered}
\left(\frac{2}{z-w^i_\ell}+\langle{\alpha}_i, A^H(z)\rangle+\partial_z\log\Lambda_i(z)\right)\Big|_{z=w^i_\ell}=0, \\ 
i=1,\dots, r, \quad  \ell=1, \dots, \deg(q^i_+(z)).
\end{gathered}
\end{equation}
We now apply gauge change by $g_{i,\ell}(z)=\exp\Big[\frac{-f_i}{\Lambda_i(z)(z-w^i_\ell)}\Big]$  to
$$\nabla =\partial_z +\sum^r_{i=1}\Lambda_i(z)e_i+ A^H(z).$$

The only terms in which $z-w^i_\ell$ appears in the denominator are those
which involve $\check\alpha_i$ and $f_i$.  The former gives $\frac{1}{z-w^i_\ell}\check\alpha_i+\langle{\alpha}_i,A^H(z)\rangle
  \check\alpha_i$, and since $\langle{\alpha}_i,A^H(z)\rangle$
  has a simple pole at $w^i_\ell$ with residue $-1$, this expression is
  regular at $w^i_\ell$.  The terms involving $f_i$ are 
\begin{equation*} -\partial_z
  (\frac{1}{\Lambda_i(z)(z-w^i_\ell)})f_i-\frac{f_i}{\Lambda_i(z-w^i_\ell)^2}-\frac{\langle{\alpha}_i,A^H(z)\rangle
  f_i}{\Lambda_i(z)(z-w^i_\ell)}=-\frac{\partial_z\log\Lambda_i+2(z-w^i_\ell)^{-1}+\langle{\alpha}_i,A^H(z)\rangle}{\Lambda_i(z-w^i_\ell)}f_i.
  \end{equation*}
 The residue term of this vanishes by the Bethe Ansatz equations
 \eqref{Bethecartan}, and we conclude that the matrix of $\nabla$ in
 this gauge is manifestly regular at $w^i_\ell$.
  
Thus, we have proved the following theorem:
\begin{Thm}\label{regular}
The nondegenerate Miura-Pl\"ucker oper connections are regular at the $\{w^i_\ell\}$'s.
\end{Thm}

\section{B\"acklund transformations}\label{S:backlund}

In this section, we show that nondegenerate $Z^H$-twisted Miura-Pl\"ucker opers
are in fact $Z^H$-twisted Miura opers.  Thus, solutions of the Bethe Ansatz equations
are in fact parameterized by $Z^H$-twisted Miura opers.  The proof
relies on the important
technical tool of \emph{B\"acklund transformations}: transformations
on twisted Miura-Pl\"ucker opers associated to elements of the Weyl
group.

\subsection{Simple B\"acklund transformations}

Our goal is to define transformations which take a $Z^H$-twisted
Miura-Pl\"ucker oper to a $w(Z^H)$-twisted Miura-Pl\"ucker oper, where
$w$ is an element of the Weyl group.  As a first step, we consider the
case of a simple reflection $s_i$.  Recall that a polynomial solution
of the $qq$-system gives rise to a connection \eqref{key} defined in
terms of the $q^j_+$'s and $Z^H$.  We now exhibit a gauge
transformation which takes this connection to another connection in
the form \eqref{key}, but with $q^i_+$ and $Z^H$ replaced by $q^i_-$
and $s_i(Z^H)$.  This gauge transformation is by an element of
$N_-(z)$, so it does not preserve the Miura-Pl\"ucker
oper structure.

\begin{Prop}    \label{fiter}  Let $\{q^j_+,q^j_-\}_{j=1,\dots,r}$ be a polynomial
  solution of the $qq$-system \eqref{qq}, and let $\nabla$ be the
  connection of the corresponding $Z$-twisted Miura-Pl\"ucker oper
  in the form \eqref{key}.
  Let $\nabla^{(i)}$ be the connection obtained
  from $\nabla$ via the gauge transformation by $e^{\mu_i(z)f_i}$,
  where
  \begin{equation}\label{eq:PropDef}
    \mu_i(z)=\Lambda_i(z)^{-1}\Bigg[\partial_z\log\Bigg(\frac{q^i_-(z)}{q^i_+(z)}\Bigg)+\langle
    \alpha_i,Z^H\rangle\Bigg]
  \end{equation}
Then $\nabla^{(i)}$ is obtained by making the following substitutions
in \eqref{key}:

\begin{equation}\begin{aligned}
q^j_+(z) &\mapsto q^j_+(z), \qquad j \neq i, \\
q^i_+(z) &\mapsto q^i_-(z), \qquad Z\mapsto s_i(Z^H)=Z^H-\langle \alpha_i, Z^H\rangle\ \check{\alpha}_i.
\label{eq:Aconnswapped}  
\end{aligned}
\end{equation}
\end{Prop}

\begin{proof}
  A short computation shows that
\begin{multline}
\nabla^{(i)}=e^{\mu_i(z)f_i}~\nabla ~ e^{-\mu_i(z)f_i}=\\
\partial_z+A^H(z)-\Lambda_i(z)\mu_i(z)\check{\alpha}_i+\sum^r_{k=1}\Lambda_k(z)e_k
+f_i\Big(\mu_i(z)\langle \alpha_i,A^H(z)\rangle-\mu'_i(z)-\mu_i(z)^2\Lambda_i(z)\Big)
\end{multline}

In this expression, the diagonal term is
$$Z^H-\langle \alpha_i,Z^H\rangle\ \check{\alpha}_i-
\sum_j\frac{\partial_z q^j_+(z)}{q^j_+(z)}\check{\alpha}_j
-\partial_z\log\left(\frac{q^i_-(z)}{q^i_+(z)}\right) \check{\alpha}_i= s_i(Z^H)
-\frac{\partial_z q^i_-(z)}{q^i_-(z)}\check{\alpha}_i-\sum_{j\ne i}\frac{\partial_z
  q^j_+(z)}{q^j_+(z)}\check{\alpha}_j$$ as desired.

Thus the statement of the theorem is true if $\mu_i$ satisfies  the
\emph{Ricatti equation}:
\begin{equation}
\frac{\mu'_i(z)}{\mu_i(z)}+\mu_i(z)\Lambda_i(z)=\langle \alpha_i, A^H(z)\rangle.
 \end{equation}
Setting $h_i(z)=\Lambda_i(z)\mu_i(z)$, this equation is equivalent to 
 \begin{equation}
 \frac{h'_i(z)}{h_i(z)}+h_i(z)=\langle \alpha_i, A^H(z)\rangle+\partial_z\log(\Lambda_i(z)).
\end{equation}

This identity now follows by taking the logarithmic derivative of
\eqref{qqred}:

\begin{equation*}
  \begin{aligned} \frac{h'_i(z)}{h_i(z)}+h_i(z)&=\partial_z\log
    h_i(z)+\partial_z\log\Bigg(\frac{q^i_-(z)}{q^i_+(z)}\Bigg)+\langle\alpha_i,Z^H\rangle z
    =\partial_z\log\left(\partial_z\Bigg[\Bigg(\frac{q^i_-(z)}{q^i_+(z)}\Bigg)
      e^{\langle\alpha_i,Z^H\rangle
        z}\Bigg]\right)\\
    &=\partial_z\log\left(\Lambda_i(z)\left[\prod_{j}q_+^j(z)^{-a_{ji}}\right] e^{\langle\alpha_i,Z^H\rangle
        z}\right)
    =\partial_z\log(\Lambda_i(z))+\langle  \alpha_i, A^H(z)\rangle.
  \end{aligned}
\end{equation*}
\end{proof}

\begin{Rem}
Note that $\mu_i(z)$ can be rewritten using the $qq$-system equations as:
\begin{equation}\label{mualt}
\mu_i(z)=\frac{\prod_{j\neq i}q^j_+(z)^{-a_{ji}}}{q_+^i(z)q_-^i(z)}.
\end{equation}
\end{Rem}

\subsection{General B\"acklund transformations}

We would like to construct B\"acklund transformations associated to an
arbitrary element $w$ of the Weyl group by taking a reduced expression
for $w$ and composing simple B\"acklund transformations associated to
the given simple reflections.  However, in general, it is not possible
to compose B\"acklund transformations.  The problem is that, even if
one starts with a nondegenerate solution of the $qq$-system, the
connection $\nabla^{(i)}$ defined in Proposition \ref{fiter} is not
necessarily the underlying connection of a nondegenerate
$s_i(Z^H)$-twisted Miura-Pl\"ucker oper.  It thus does not give rise
to the necessary initial data for another B\"acklund transformation,
namely a solution of the $qq$-system for $s_i(Z^H)$.

\begin{Def}  Let $\{q^j_+(z), q^j_-(z)\}$ be a polynomial solution of
  the $qq$-system for $Z^H$.
  \begin{enumerate}\item The solution is called
  \emph{$i$-composable} if the 
  polynomials $q^1_+(z),\dots, q^{i-1}_+(z), q^{i}_-(z),
  q^{i+1}_+(z),\dots, q^{r}_+(z)$ are the positive polynomials of a
  solution to the $qq$-system for $s_i(Z^H)$.
\item The solution is called \emph{$i$-generic} if it is nondegenerate
  and if the collection of polynomials
  $q^1_+(z),\dots, q^{i-1}_+(z), q^{i}_-(z), q^{i+1}_+(z),\dots,
  q^{r}_+(z)$ satisfy the conditions in Definition~\ref{nondeg
    Cartan}.
\end{enumerate}
\end{Def}

We will also refer to a twisted Miura-Pl\"ucker datum as
$i$-composable or $i$-generic if it comes from such a solution of the
$qq$-system. 

It is immediate from Proposition~\ref{fiter} and Theorem~\ref{inj}
that if $\{q^j_+(z), q^j_-(z)\}$ is $i$-composable, then
$\nabla^{(i)}$ is the underlying connection of a $s_i(Z^H)$-twisted
Miura-Pl\"ucker oper.

\begin{Rem}\label{remgeneric} Assume that
  $\mathcal{Q}_{d_1,\dots,d_r}$ is nonempty.  While it is easy to see
  that $i$-genericity is a Zariski-open condition on the variety
  $\mathcal{Q}_{d_1,\dots,d_r}$, it is not clear that this open subset
  is nonempty.  In other words, $q^{i}_-(z)$ may have multiple roots
  or it may share a root with $\Lambda_i(z)$ or with $q^j_+(z)$ for
  $j\ne i$ such that $a_{ji}\ne 0$.  However, if
  $\langle \alpha_i,Z^H\rangle= 0$, the set of $i$-generic polynomial
  solutions is nonempty.  Indeed, if $q^{i}_-(z)$ does not satisfy the
  conditions in Definition~\ref{nondeg Cartan}, one can replace it by
  $q^{i}_-(z)+cq^i_+(z)$ for an appropriate nonzero scalar $c$.  In
  particular, when $Z=0$, the set of nondegenerate polynomial
  solutions of the $qq$-system that are $i$-generic for all $i$ is
  nonempty.

\end{Rem}

\begin{Lem} \label{nondegcond} If $\{q^j_+,q^j_-\}_{j=1,\dots,r}$ is
  an $i$-generic polynomial
  solution of the $qq$-system, then it is $i$-composable.  In particular,
  \begin{enumerate}\item The connection $\nabla^{(i)}$
  constructed in Proposition~\ref{fiter} is the underlying connection
  of a nondegenerate $s_i(Z^H)$-twisted Miura-Pl\"ucker oper.
  \item Any
  corresponding (necessarily nondegenerate) polynomial  solution
  $\{\wt{q}^j_+,\wt{q}^j_-\}_{j=1,\dots,r}$ of the $qq$-system for
  $s_i(Z^H)$ has
  $\wt{q}^i_+=q^i_-$ and $\wt{q}^j_+=q^j_+$ for $j\ne i$.  Moreover,
  one may take $\wt{q}^i_-=-q^i_+$.
\end{enumerate}
\end{Lem}

\begin{proof}

We will show that the polynomials $\{\wt{q}^j_+\}$ defined above give
rise to a solution of the Bethe Ansatz equations
for $s_i(Z^H)$.  It will then follow from Theorem~\ref{BAE} that there
exist polynomials $\wt{q}^j_-$ such that
$\{\wt{q}^j_+,\wt{q}^j_-\}_{j=1,\dots,r}$ is a nondegenerate
polynomial solution of the $qq$-system; moreover, this solution will
correspond to a $s_i(Z^H)$-twisted Miura-Pl\"ucker datum with
underlying connection $\nabla^{(i)}$.  We will show explicitly that
one can take $\wt{q}^i_-=-q^i_+$.

First, note that $W(q^i_-,-q^i_+)=W(q^i_+,q^i_-)$ and $\langle   \alpha_i,s_i(Z^H)\rangle
{q^i_-(z)(-q^i_+(z))}=\langle   \alpha_i,Z^H\rangle
{q^i_+(z)q^i_-(z)}$.  It is now immediate that the $i$th 
equation of  the $qq$-system for $s_i(Z^H)$ is satisfied by the
$\wt{q}^j_+$'s and $\wt{q}^i_-=-q^i_+$.  As in the proof of
Theorem~\ref{BAE}, this implies that the 
Bethe Ansatz equations \eqref{bethe} involving the roots of
$\wt{q}^i_+=q^i_-$ are satisfied.

Next, rewrite the $i$th equation of the original $qq$-system as

\begin{equation}    \label{qqre}
  \partial_z\log(q^i_-(z))-\partial_z\log(q^i_+(z))+\langle  Z^H,
\alpha_i\rangle =\frac{\Lambda_i(z)}{q^i_+(z)q^i_-(z)}\prod_{j\ne i}\Big [ q_+^j(z)\Big ]^{-a_{ji}}.
\end{equation}
  Evaluating this expression at a root $w_\ell^j$ of $q^j_+(z)$ for $j\ne
i$ and using nondegeneracy, one obtains
\begin{equation}
\partial_z\log(q^i_-(z))\Big|_{w^j_\ell}+\langle  \alpha_i,Z^H\rangle=\partial_z\log(q^i_+(z))\Big|_{w_\ell^j}.
\end{equation}

One gets the remaining Bethe Ansatz equations by substituting this
into \eqref{bethe}:
\begin{equation}
\begin{aligned}
  0&=\langle  \alpha_j,Z^H\rangle+\partial_z\log\Big[\Lambda_j(z)\prod_{k}q_+^k(z)^{-a_{kj}}(z-w^j_\ell)^2\Big]\Bigg|_{z=w^j_\ell}\\
  &=\langle  \alpha_j,Z^H\rangle-a_{ij}\langle  \alpha_i,Z^H\rangle+\partial_z\log\Big[\Lambda_j(z) q_-^i(z)^{-a_{ij}}\prod_{k\ne
    i}q_+^k(z)^{-a_{kj}}(z-w^j_\ell)^2\Big]\Bigg|_{z=w^j_\ell}\\
  &=\langle  \alpha_j,s_i(Z^H)\rangle+\partial_z\log\Big[\Lambda_j(z)\prod_{k}\wt{q}_+^k(z)^{-a_{kj}}(z-w^j_\ell)^2\Big]\Bigg|_{z=w^j_\ell}.
\end{aligned}
\end{equation}

\end{proof}

Thus, the $i$th simple B\"acklund transformation may be viewed as
taking an $i$-generic Miura-Pl\"ucker datum to a nondegenerate
$s_i(Z^H)$-twisted Miura-Pl\"ucker oper.

\begin{Def}\label{genericdef}
  Let $w=s_{i_1} \dots s_{i_k}$ be a reduced decomposition of an
  element $w$ of the Weyl group. \begin{enumerate}\item  A polynomial solution of the $qq$-system
  \eqref{qq} for $Z^H$ is called \emph{$({i_1}, \dots,
    {i_k})$-composable} 
  if for each $\ell$, $1\le\ell\le k$, the connection
   $\nabla^{(i_k)\dots(i_{k-\ell+1})}$ comes from a polynomial solution
     of the $qq$-system for
     $s_{i_{k-\ell+1}}\dots s_{i_k}(Z^H)$.
   \item The solution is called
     \emph{$({i_1}, \dots, {i_k})$-generic}) if it is nondegenerate
     and for each $\ell$, $1\le\ell\le k$, the connection
     $\nabla^{(i_k)\dots(i_{k-\ell+1})}$ comes from a nondegenerate
     polynomial solution of the $qq$-system for
     $s_{i_{k-\ell+1}}\dots s_{i_k}(Z^H)$.
     \item A $Z^H$-twisted  Miura-Pl\"ucker oper is called \emph{$({i_1}, \dots,
    {i_k})$-composable} (resp. \emph{$({i_1}, \dots, {i_k})$-generic})
  if it arises from such a solution of the $qq$-system.
 \end{enumerate}
\end{Def}

It is immediate that $({i_1}, \dots, {i_k})$-genericity implies
$({i_1}, \dots, {i_k})$-composability.

\begin{Rem}
  Note that in this definition, we only assume the existence of a
  sequence of transformations as described in Lemma \ref{nondegcond}
  for a particular reduced decomposition of $w$. We do not assume that
  such a sequence exists for other reduced decompositions of $w$.
\end{Rem}

We will need a technical result for $({i_1} \dots {i_k})$-composable
solutions of the $qq$-system, showing the existence of an element of
$B_-(z)$ which intertwines the action of $\nabla$ and $s_{i_1} \dots
s_{i_k}(Z^H)$ on highest weight vectors.

\begin{Prop} \label{bminus} Let $w=s_{i_1} \dots s_{i_k}$ be a reduced
  decomposition.  Then, for each $({i_1} \dots {i_k})$-composable
  solution of the $qq$-system \eqref{qq}, there exists an element
  $b_-(z)\in B_-(z)$ of the form
$$
b_-(z)=e^{c_{i_k}(z)f_{i_k}}\dots e^{c_{i_2}(z)f_{i_2}}
  e^{c_{i_1}(z)f_{i_1}}h(z),
$$
where $c_{i_j}(z)$ are non-zero rational functions and $h(z)\in H(z)$,
such that
\begin{equation}
b_-(z)w(Z^H)v=\partial_zb_-(z)v+A(z)b_-(z)v.
\label{eq:OperatorDiffbm}
\end{equation}
Here, $A(z)$ is given by equation \eqref{key}
and $v$ is a highest weight vector in any irreducible
finite-dimensional representation of $G$.
\end{Prop}

\begin{proof}

Let $\nabla^w$ be the $w(Z^H)$-twisted Miura-Pl\"ucker
oper obtained by iterating the B\"acklund transformations defined in Proposition~\ref{fiter}:
\begin{equation} \label{gaugew} \nabla^w=
  e^{\mu_{i_1}(z)f_{i_1}}\dots e^{\mu_{i_k}(z)f_{i_k}}\nabla e^{-\mu_{i_k}(z)f_{i_k}}\dots e^{-\mu_{i_1}(z)f_{i_1}}.
\end{equation}
Let $\{\bar{q}^i_+\}_{i=1,\dots,r}$ be the ``plus'' part of the
corresponding solution to the $qq$-system.  We claim that
\begin{equation}
b_-(z)=e^{-\mu_{i_k}(z)f_{i_k}}\dots
e^{-\mu_{i_1}(z)f_{i_1}}\prod_j
\Big[\overline{q}^j_+(z)\Big]^{\check{\alpha}_j}
\end{equation}
satisfies \eqref{eq:OperatorDiffbm}.

Let $V$ be an irreducible representation with highest weight
$\lambda$, and let $v\in V$ be a highest weight vector. First observe
that
\begin{equation}
\nabla^{w}v=w(Z^H)v-\left(\sum^r_{j=1}
 \frac{\partial_z{\bar{q}^j_+(z)}}
  {\bar{q}^j_+(z)}\check{\alpha}_j\right)v.
\end{equation}

For brevity, write $E(z)=e^{-\mu_{i_k}(z)f_{i_k}}\dots
e^{-\mu_{i_1}(z)f_{i_1}}$.  We now compute:
\begin{align*}
  (\partial_z + A(z) )b_-(z)v&=\prod_j
\Big[\bar{q}^j_+(z)\Big]^{\langle\check{\alpha}_j,\lambda\rangle}(\partial_z
  + A(z))E(z)v+b_-(z)\left(\sum_j
 \frac{\partial_z{\bar{q}^j_+(z)}}
  {\bar{q}^j_+(z)}\check{\alpha}_j\right) v\\
  &=\prod_j
\Big[\bar{q}^j_+(z)\Big]^{\langle\check{\alpha}_j,\lambda\rangle}E(z)\nabla^w
    v+b_-(z)\left(\sum_j
 \frac{\partial_z{\bar{q}^j_+(z)}}
    {\bar{q}^j_+(z)}\check{\alpha}_j\right)v\\
    &=b_-(z)\left[w(Z^H)v-\left(\sum^r_{j=1}
 \frac{\partial_z{\bar{q}^j_+(z)}}
  {\bar{q}^j_+(z)}\check{\alpha}_j\right) v\right]+b_-(z)\left(\sum_j
 \frac{\partial_z{\bar{q}^j_+(z)}}
      {\bar{q}^j_+(z)}\check{\alpha}_j\right) v\\
  &=b_-(z)w(Z^H)v,
\end{align*}
as desired.

\end{proof}

\subsection{$Z^H$-twisted Miura-Pl\"ucker opers with admissible
  combinatorics are  $Z^H$-twisted Miura opers}    \label{main thm}

We now prove one of the main results of the paper, namely, that
$Z^H$-twisted Miura-Pl\"ucker opers satisfying certain combinatorial
conditions are in fact nondegenerate $Z^H$-twisted Miura opers.  We
begin by outlining the argument.

The first step is to define a class of $Z^H$-twisted Miura-Pl\"ucker
opers for which one can give an explicit construction of an upper
triangular matrix which diagonalizes the oper, thereby showing that it
is a $Z^H$-twisted Miura oper.  The desired condition will be called
$w_0$-{\em genericity} (or more generally, $w_0$-{\em composability});
it will be a special case of the genericity considered in
Definition~\ref{genericdef}.


Next, we observe that the behavior of the $qq$-system and its iterates under
B\"acklund transformations depend on certain underlying combinatorics:
the set of roots killing $Z^H$, the degrees of
the $\Lambda_i$'s, and the degrees of the $q^i_+$'s.   This
combinatorial data essentially determine the degrees of the $q_-^i$'s
and inductively, the degrees of the polynomials appearing as solutions
of the new $qq$-systems obtained after applying B\"acklund
transformations.  We will call this combinatorial data
\emph{admissible} if there exists a $w_0$-generic solution of the
$qq$-system with the given combinatorics.

Finally, we show that twisted Miura-Pl\"ucker opers with admissible
combinatorics are in fact Miura opers.  To do this, we introduce
formal variables associated to the given admissible combinatorics: for
the coordinates of a certain affine variety determined by the set of roots, for the zeros of the
$q_+^i$'s and other $\tilde{q}_{\pm}^i$'s that appear upon an
appropriate iteration of B\"acklund transformations, and for the zeros
of the $\Lambda_i$'s.  We construct a ring $R$ by adjoining these
formal variables to $\C(z)$ and taking a suitable localization.  One
can now define a $qq$-system $\{Q^i_+,Q^i_-\}$ over $R$ which has the
property that upon specializing the formal variables appropriately,
one obtains an ordinary $qq$-system with the given combinatorics.
Moreover, $\{Q^i_+,Q^i_-\}$ is $w_0$-generic because it specializes to
an ordinary $w_0$-generic $qq$-system.  We can use this fact to deduce
that Miura-Pl\"ucker opers with the given combinatorics are in fact
Miura opers.

\subsubsection{$w_0$-composability and $w_0$-genericity}

We begin by describing a sufficient condition for a $Z^H$-twisted
Miura-Pl\"ucker oper to be a $Z^H$-twisted Miura oper.  Let $w_0$ be
the longest element of the Weyl group.  We call a solution of the
$qq$-system (or the corresponding Miura-Pl\"ucker oper) $w_0$-generic
(resp. \emph{$w_0$-decomposable}) if there exists a reduced
decomposition $w_0=s_{i_1}\dots s_{i_{\ell}}$ such that the solution
(or oper) is $(i_1,\dots,i_{\ell})$-generic (resp. composable).  (For
any $w\in W$, one defines $w$-genericity and $w$-composability
similarly.)


We will need the following well-known fact about the product of
Bruhat cells (see e.g. \cite[Lemma 29.3.A]{humphreys}):

\begin{Lem} \label{red}
i) If $u, v\in W$ satisfy $\ell(u)+\ell(v)=\ell(uv)$, then
$B_-uB_-vB_-=B_-uvB_-$.\\
ii) If $w\in W$ has a reduced decomposition $w=s_{i_1}s_{i_2}\dots
s_{i_k}$, then $$e^{a_{i_1}e_{i_1}}
e^{a_{i_2}e_{i_2}}\dots e^{a_{i_k}e_{i_k}}\in B_-wN_-, \quad e^{a_{i_1}f_{i_1}}
e^{a_{i_2}f_{i_2}}\dots e^{a_{i_k}f_{i_k}}\in B_+wN_+$$ if
$a_{i_j}\neq 0$ for all $j$.
\end{Lem}

\begin{Thm}    \label{w0}
Every $w_0$-composable (resp. $w_0$-generic) $Z^H$-twisted Miura-Pl\"ucker $G$-oper is a
$Z^H$-twisted Miura $G$-oper (resp. a nondegenerate $Z^H$-twisted Miura $G$-oper).  
\end{Thm}

\begin{proof} 
Let  $$
\nabla=\partial_z+A(z)=\partial_z+\sum^r_{i=1}\Big [\zeta_i -{q^i_+}(z)^{-1}\partial_z{q^i_+(z)}\Big]\check{\alpha}_i+\sum^r_{i=1}\Lambda_i(z)e_i
$$ be the $w_0$-composable $Z^H$-twisted Miura-Pl\"ucker oper
coming from a $w_0$-composable solution $\{q^j_{+}\}$ of the $qq$-system.
By Proposition \ref{bminus}, there exists an element $b_-(z)\in
B_-(z)$ such that 
$$
b_-(z)w_0(Z^H)v=(\partial_z+A(z))b_-(z)v,
$$
where $v$ is any highest weight vector in a finite-dimensional
irreducible representation of $G$.  Moreover, if 
$w_0=s_{i_1}\dots s_{i_\ell}$ is a reduced expression for which the
solution is $(i_1,\dots,i_\ell)$-composable, then 
$$b_-(z)=e^{c_{i_\ell}f_{i_\ell}}\dots e^{c_{i_2}f_{i_2}}e^{c_{i_1}f_{i_1}}h(z)
$$
with $c_{i_j}(z) \in \C(z)^\times$ and $h(z)\in H(z)$.

By Lemma~\ref{red} and the fact that $w_0$ is an involution,
$$
b_-(z)=b_{+}(z)w_0 n_+(z),
$$
for some $b_+(z)\in B_+(z)$ and $n_+(z)\in N_+(z)$, so if $v$ is a
highest weight vector in an irreducible representation,
$$
b_+(z)Z^Hw_0v=(\partial_z+A(z))b_+(z)w_0 v.
$$
Therefore, if we set
\begin{equation}    \label{Uz}
u(z)=Z^H-b^{-1}_+(z)\partial_zb_+(z)+b^{-1}_+(z)A(z)b_+(z)\in \mathfrak{b}_+(z),
\end{equation}
then
$$
u(z)w_0 v=0.
$$
for any irreducible finite-dimensional representation of $G$ with
highest weight vector $v$. Thus, $u(z)$ is an element of $\mathfrak{b}_+(z)$
which fixes the lowest weight vector $w_0 v$ of any irreducible
finite-dimensional representation of $G$. This means that $u(z)=0$.
Equation \eqref{Uz} then implies that $A(z)$ satisfies
\begin{equation}    \label{Abplus}
A(z) = b_+(z) (\partial_z+Z^H) b_+(z)^{-1}
\end{equation}
for some $b_{+}(z) \in B_+(z)$.  Thus, we have proved that every
$w_0$-composable $Z^H$-twisted Miura-Pl\"ucker oper is a
$Z^H$-twisted Miura oper. Equivalently, every
$w_0$-composable solution of the $qq$-system gives rise to a
$Z^H$-twisted Miura oper.  By definition, if the original solution
is in fact $w_0$-generic, then the corresponding $Z^H$-twisted Miura
oper is nondegenerate.
\end{proof}

\subsubsection{Admissible combinatorial data}

Let $d_1,\dots,d_r$ and $N_1,\dots,N_r$
be nonnegative integers, and let $\Psi$ be a collection of roots.  Set
$\fh_\Psi=\{Y\in \fh\mid \beta(Y)=0\iff \beta\in\Psi\}$; it is an
affine cone.

\begin{Def} The combinatorial datum $(\mathbf{d}=(d_1,\dots,d_r),\mathbf{N}=(N_1,\dots,N_r),\Psi)$ is called \emph{$w_0$-admissible} (or simply
  \emph{admissible}) if there exists a $w_0$-generic solution of the
  $qq$-equations with $Z_H\in \fh_\Psi$ and for all $i$,
  $\deg \Lambda_i= N_i$ and $\deg q^i_+=d_i$.
\end{Def}

\begin{Rem}  One may similarly define $w$-admissibility.  In this
  language, $e$-admissibility combinatorics simply means that there exists a
  nondegenerate polynomial solution with the given combinatorics.
\end{Rem}

We now give a more explicit formulation of admissibility in the two opposite
extremes $Z^H=0$ and $Z^H$ regular semisimple, i.e. $\Psi$ equals
$\Phi$ (the set of all roots) or $\varnothing$.

\begin{Prop} The combinatorial datum  $(\mathbf{d},\mathbf{N},\Phi)$
  is admissible if and only if there exists a nondegenerate polynomial solution of
  the $qq$-system with $Z^H=0$ and for all $i$, $\deg q^i_+(z) =d_i$
  and $\deg \Lambda_i=N_i$.
\end{Prop}
\begin{proof}
  By induction, it suffices to show that for any nondegenerate
  solution and for any $i$, one can modify $q^i_-(z)$ so that the
  solution is $i$-generic.  This was shown in Remark~\ref{remgeneric}.
\end{proof}

We now assume that $Z^H$ is regular semisimple.  In this case, one can
characterize admissibility explicitly in terms of certain inequalities
that must be satisfied by the $d_j$'s and $N_j$'s.

We first observe that a B\"acklund transformation induces a
transformation on the set of $d_j$'s.  Indeed, as we have seen in
Remark~\ref{degrees},  if $Z^H$ is regular
semisimple, then $\deg
  q^i_-=\deg\Lambda_i-\deg q^i_+-\sum_{j\ne i}a_{ji}\deg q^j_+$.
  Accordingly, the $i$th B\"acklund transformation takes $d^j\mapsto
  d^{(i)}_j$, where
  \begin{equation} d^{(i)}_j=\begin{cases} N_i-d_i-\sum_{k\ne i}a_{ki}
      d_k & \text{if $j=i$}\\ d_j  & \text{otherwise.}
    \end{cases}
  \end{equation}

 The following necessary condition for the existence of an
 $(i_1,\dots,i_k)$-composable solution of the $qq$-system with fixed regular semisimple
 combinatorics is now immediate.

 \begin{Lem}  If there exists an $(i_1,\dots,i_k)$-composable
   polynomial solution with
   combinatorial datum $(\mathbf{d},\mathbf{N},\varnothing)$,
   then for $0\le s\le k$\footnote{By
       convention, the case $s=0$ corresponds to the original $d_j$'s.} and $1\le j\le r$,
   \begin{equation}\label{inequalities} d^{(i_k)\dots(i_{k-s+1})}_j\le N_j -\sum_{\ell\ne
       j}a_{p j}d^{(i_k)\dots(i_{k-s+1})}_p.
   \end{equation}
 \end{Lem}

 It turns out that if $\fg$ is simply-laced, then this necessary
 condition is in fact sufficient.  Moreover, one can find a generic
 solution with the given combinatorics.  In order to prove this, we
 will consider a limit of the $qq$-system, the \emph{infinite
   $qq$-system}.

 Let $\xi_i=\langle \alpha_i,Z^H\rangle$.  To take the limit of the
 $i$th $qq$-equation as $\xi$ goes to infinity, we need to rewrite the
 equation.  Since the right-hand side of the equation is monic, we have
 $q^i_-(z)=\xi_i^{-1}\bar{q}^i_-(z)$, where $\barq^i_-(z)$ is
 monic.  The $i$th $qq$-equation is thus equivalent to 
 \begin{equation}
   {\xi_i^{-1} W(q^i_+,{\barq}^i_-)(z)}+
{q^i_+(z)\barq^i_-(z)}=\Lambda_i(z)\prod_{j\neq i}\Big [ q_+^j(z)\Big
]^{-a_{ji}}
\end{equation}
Upon taking the limit, the Wronskian term disappears. 

 \begin{Def}  The \emph{infinite $qq$-system} associated to $\mathfrak{g}$
  and the collection of monic polynomials $\Lambda_1(z),\dots,\Lambda_r(z)$ is
  the system of equations
\begin{equation}    \label{qqinfinite}
{q^i_+(z)q^i_-(z)}=\Lambda_i(z)\prod_{j\neq i}\Big [ q_+^j(z)\Big
]^{-a_{ji}}\qquad \text{for $i=1,\dots,r$,}
\end{equation}
 where the $q^j_+(z)$'s (and hence the $q^i_-(z)$'s) are assumed to be
 monic.
 \end{Def}

 It is easier to understand the significance of the infinite
 $qq$-system in the $q$-deformed case~\cite{Frenkel:2020}. The
 $q$-difference analog of the $qq$-system, known as the $QQ$-system,
 expresses the relations between the so-called Baxter $Q$-operators in
 the corresponding XXZ integrable model \cite{Frenkel:2013uda},
 \cite{Frenkel:ac}, acting on a tensor product $\mathcal{H}$ of
 finite-dimensional representations of the quantum group
 $U_q(\widehat{\mathfrak{g}})$.  (This tensor product is the
 underlying Hilbert space of the XXZ model).  The Baxter $Q$-operators
 can be expressed as weighted half-traces
 $Q_{\pm}^i(z)=\Tr_{V_{\pm}^i}\Big[(\mathcal{Z}\otimes I)R\Big]$ in
 the so-called prefundamental representations
 $\{V_{\pm}^i\}_{i=1,\dots,r}$ of $U_q(\widehat{b}_+)$ (see \cite{HJ})
 of the normalized universal $R$-matrix
 $R \in U_q(\widehat{b}_+)\widehat{\otimes} U_q(\widehat{b}_-)$; here,
 the weight $\mathcal{Z}=\prod_i\hat{\zeta}_i^{\check{\alpha}_i}$ is a
 deformation of the classical $Z^H$. The $Q$-operators act on
 $\mathcal{H}$ through the second factor of the $R$-matrix, i.e.,
 through $U_q(\widehat{b}_-)\subset U_q(\widehat{\mathfrak{g}})$.

 One can define the infinite version of such Baxter $Q$-operators by
 considering the limit as the corresponding multiplicative weight
 parameters $\hat{\xi}_i=\prod_j \hat{\zeta_j}^{-a_{ji}}$ goes to
 zero. One can even write an explicit formula expressing the expansion
 coefficients of the $Q$-operator in terms of their infinite analogues
 and the generators of the quantum group. This was done explicitly in
 \cite{Frenkel:2013uda} and \cite{Pushkar:2016qvw} in the case of
 $\mathfrak{g}=\mathfrak{sl}(2)$. The latter reference, together with
 the subsequent works \cite{Koroteev:2017aa}, \cite{KZtoroidal},
 \cite{KZ3d}, identified the infinite version of the $QQ$-system
 relations with the relations in the classical equivariant $K$-theory
 ring on a certain quiver variety while the finite version gives the
 relations in the quantum $K$-theory ring.  The parameters
 $\hat{\xi}_i$ are known as K\"ahler parameters.

 In particular, these results for Baxter $Q$-operators imply that one
 can find solutions of the $QQ$-system which are the deformations of
 solutions of its infinite analogue. Upon taking the limit which
 reduces the XXZ model to Gaudin model (see e.g. Section 6 of
 \cite{KSZ}), we see that this is true for the $qq$-system as well.

For example, in the $\mathfrak{sl}(2)$ case,  the infinite $qq$-system is simply the single equation
 $q_+(z)q_-(z)=\Lambda(z)$.  If we set
 $\Lambda(z)=\prod_{j=1}^N (z-z_j)$, then a solution is obtained by
 dividing the $z_j$'s into $w_1,\dots,w_d$ and $v_1,\dots,v_{N-d}$ and
 setting $q^\infty_+(z)=\prod_{k=1}^d (z-w_k)$ and
 $q^\infty_-(z)=\prod_{\ell=1}^{N-d} (z-v_\ell)$.  Then following the discussed above $q$-deformed case, if  $\Lambda(z)$ has no
 repeated roots, then for large enough $\xi$, there are deformations
 $w_k^\xi$ and $v_\ell^\xi$ such that $q^{\xi}_+(z)=\prod_{k=1}^d
 (z-w^\xi_k)$, $q^{\xi}_-(z)=\prod_{\ell=1}^{N-d}
 (z-v^\xi_\ell)$ are a solution of the finite $qq$-system (for the
 same $\Lambda(z)$) with parameter $\xi$.  Moreover, given the initial
 choice of $q^\infty_+(z)$, the solution is unique and is indeed given by a
 formula that allows one to view the $z_j$'s as free parameters.  

 \begin{Lem}  If $\Lambda(z)$ has no repeated roots, then the finite
   solution $q^{\xi}_+(z), q^{\xi}_-(z)$ is nondegenerate for large $\xi$.
 \end{Lem}
 \begin{proof} Since $q^\infty_+(z)$ has no multiple roots and is
   relatively prime to $q^\infty_-(z)$, the same holds for the finite
   solutions for large $\xi$.  For such $\xi$, suppose that
   $q^\xi_+(z)$ has a root $w$ in common with $\Lambda(z)$.  We see
   that $\partial_z q^\xi_+(z) q^\xi_-(z)$ vanishes at $w$, since
   every other term in the $qq$-equation vanishes.  This implies that
   either $w$ is either a root of $q^\xi_-(z)$ or a multiple root of
   $q^\xi_+(z)$, a contradiction.
 \end{proof}

 We can generalize this procedure to define generic solutions to the
 $qq$-system for simply-laced $\fg$.  Assume that
 $d_j\le N_j -\sum_{\ell\ne j}a_{kj}d_k$ for all $j$.  We can then
 choose $Z_j, W_j\subset\C$ such that $|Z_j|=N_j$, $|W_j|=d_j$, the
 $Z_j$'s are pairwise disjoint, $Z_j\cap W_k=\varnothing$ unless
 $a_{jk}\ne 0$, and $W_j\subset Z_j\cup\bigcup_{a_{kj}<0}W_k$.  Let
 $V_j= Z_j\cup\bigcup_{a_{kj}<0}W_k\setminus W_j$.  Setting
 $\Lambda_j(z)=\prod_{a\in Z_j} (z-a)$,
 $q^{j,\infty}_+(z)=\prod_{w\in W_j} (z-w)$, and
 $q^{j,\infty}_-(z)=\prod_{v\in V_j} (z-v)$ gives a solution of
 \eqref{qqinfinite}.  Since $\fg$ is simply-laced,
 $\Lambda_j(z)\prod_{k\ne j}q^{k,\infty}_+(z)^{-a_{kj}}$ is
 multiplicity-free.  One can now apply the results above to obtain
 unique deformations $q^{i,Z^H}_+(z), q^{i,Z^H}_-(z)$ satisfying the
 $qq$-equations.  By the lemma, these are nondegenerate solutions for
 large $Z^H$.


 Suppose further that the system of inequalities
 $d^{(i)}_j\le N_j -\sum_{\ell\ne j}a_{kj}d^{(i)}_k$ is also
 satisfied.  This guarantees that we have a solution
 $q^{j,\infty,(i)}_+(z), q^{j,\infty,(i)}_-(z)$ to the infinite
 $qq$-system with the same $\Lambda_i(z)$, with
 $q^{i,\infty,(i)}_+(z)=q^{j,\infty}_-(z)$ and
 $q^{i,\infty,(i)}_-(z)=q^{j,\infty}_+(z)$, and with
 $q^{j,\infty,(i)}_+(z)=q^{j,\infty}_+(z)$ for $j\ne i$.   We again
 can deform this infinite solution
 to obtain a unique solution of the
 $qq$-equations for $s_i(Z^H)$.  By uniqueness,
 these finite solutions are the $i$th B\"acklund transformation of the
 previous solutions, i.e., they are just
 $q^{j,Z^H,(i)}_\pm(z),q^{j,Z^H,(i)}_+(z)$.  Since these solutions are
 nondegenerate for large $Z^H$, we see that $\{q^{j,Z^H}_+(z),
 q^{j,Z^H}_-(z)\}$ is $i$-generic for large $Z^H$.

 It is clear that we can iterate this process, so we obtain the
 following theorem:

 \begin{Thm} Suppose that $\fg$ is simply-laced.  Then, there exists
   an $(i_1,\dots,i_k)$-generic solution of the $qq$-equations with
   combinatorial datum $(\mathbf{d},\mathbf{N},\varnothing)$ if and
   only if the system of inequalities \eqref{inequalities} are
   satisfied.  In particular, $(\mathbf{d},\mathbf{N},\varnothing)$ is
   admissible if and only if the system of inequalities is satisfied
   for some reduced decomposition of $w_0$.
 \end{Thm}

\subsubsection{Removing the hypothesis of $w_0$-genericity}

We now show that the $w_0$-genericity hypothesis in Theorem~\ref{w0} is
unnecessary as long as the combinatorial datum is admissible.

\begin{Thm}\label{thm:MPisM}
  Every nondegenerate $Z^H$-twisted Miura-Pl\"ucker oper with
  admissible combinatorics is a (nondegenerate) $Z^H$-twisted Miura
  oper.  In particular, this is the case when
  \begin{enumerate}\item $Z^H=0$ and there exists a nondegenerate
    polynomial solution of the $qq$-system with degrees
    $(\mathbf{d},\mathbf{N})$, and
  \item $\fg$ is simply-laced, $Z^H$ is regular semisimple, and the
    system of inequalities \eqref{inequalities} is satisfied for some
    reduced decomposition of $w_0$.
    \end{enumerate}
  \end{Thm}
\begin{proof} Let $\nabla=\partial_z+A(z)$ be a nondegenerate $Z^H$-twisted
  Miura-Pl\"ucker oper with admissible combinatorial datum
  $(\mathbf{d},\mathbf{N},\Psi)$ and with corresponding polynomials
  $q^i_+(z)$'s.  We must show
  the existence of $v(z)\in B_+(z)$ such that
  $A(z)=v(z)(\partial_z+Z^H)v(z)^{-1}$.  We will accomplish this by
  considering a solution to the $qq$-system over a ring $R$ defined in terms of certain
  formal variables.

 Let $w_0=s_{i_1}\dots s_{i_\ell}$ be a reduced decomposition for
 which there exists an $(i_1,\dots,i_\ell)$-generic solution of the
 $qq$-equations.  We now introduce formal variables for the roots of
 various polynomials:
 the $\Lambda_i(z)$'s, the positive
 polynomials $\tilde{q}^j_+(z)$'s one obtains by iterating B\"acklund transformations along
 this reduced word, and the negative polynomials
 $\tilde{q}^{i_s}_-(z)$ corresponding to the
 simple reflection at each step.  (All of these degrees are uniquely
 determined except for possibly the degrees of the
 $\tilde{q}^{i_s}_-(z)$'s.  However, we can always choose the degree
 to be the generic one specified in Remark~\ref{degrees} while
 maintaining nondegeneracy.)  Thus, we have the formal variables
 \begin{itemize} \item $\{\mathbf{z}^i_k\}$ for $1\le i\le r$ and $1\le
   k\le N_i$;
   \item $\{\mathbf{w}^{j,s}_k\}$ for $0\le s\le \ell-1$, $1\le j\le r$, and $1\le
  k\le d^{(i_\ell)\dots(i_{\ell-s+1})}_j$; and 
 \item $\{\mathbf{v}^{i_{\ell-s},s}_{k}\}$ with $1\le j\le r$ and $k$ less
 than the generic degree of $q^{(i_\ell)\dots(i_{\ell-s+1}), i_{\ell-s}}_-(z)$.
\end{itemize}

Let $R$ be the ring
$\C[\fh_\Psi]\otimes\C(z)[[\{\mathbf{w}^{j,s}_k\},
\{\mathbf{v}^{i_{\ell-s},s}_{k}\},\{\mathbf{z}^i_k\}]$, localized at
the $(z-\mathbf{w}^{i_{\ell-s},s}_k)$'s, the
$(z-\mathbf{v}^{i_{\ell-s},s}_{k})$'s, the
$(\mathbf{w}^{i,0}_k-\mathbf{z}^i_j)$'s, and the
$(\mathbf{w}^{i,0}_k-\mathbf{w}^{j,0}_s)$'s, and satisfying the Bethe
equations \eqref{BAEexplicit}.  Set
$Q^i_+(z, \{\mathbf{w}^{i,0}_k\})=\prod_k(z-\mathbf{w}^{i,0}_k)$.  We
view the $Q^i_+$'s as the ``plus'' polynomials of a $qq$-system
defined over $R$ (with the twist parameter given by a generic
$\bm{Z^H}=\sum^r_{i=1}\bm{\zeta}_i\check\alpha_i$ and the
singularities given by
$\bm{\Lambda}_i=\prod(z-\mathbf{z}^i_j)$'s). Note that this data
specializes to the data for our original
$\nabla$. 

By Theorem~\ref{BAE}, we can complete the  $Q^i_+$'s to a solution
$\{Q^i_+, Q^i_-\}$ of this $qq$-equation over $R$. 
This solution corresponds to the connection
$\partial_z+\mathcal{A}(z, \{\mathbf{w}^i_k\}, \{\bm{\zeta}_i\}, \{\bm{\Lambda}_i\})$, where
   $$\mathcal{A}(z, \{\mathbf{w}^i_k\}, \{\bm{\zeta}_i\}, \{\bm{\Lambda}_i\})=\sum^r_{i=1}\Big [\bm{\zeta}_i
  -{Q^i_+}(z)^{-1}\partial_z{Q^i_+(z)}\Big]\check{\alpha}_i+\sum^r_{i=1}\bm{\Lambda}_i(z)e_i.$$
  Again, our original connection $\nabla$ is a specialization of this
  connection.

We claim that the $qq$-system $\{Q^i_+,Q^i_-\}$ over $R$ is
$w_0$-generic.  To see this, it suffices to show that some
specialization of this $qq$-system is $w_0$-generic.  This exists by
the definition of admissibility.

Note that in the definition of the $i$th B\"acklund transformation,
$\mu_i$'s (see \eqref{mualt}) is a rational function with
$q^i_+(z) q^i_-(z)$ in the denominator.  It follows that all the
$\mu_i$'s needed in iterating B\"acklund transformations for
$\{Q^i_+,Q^i_-\}$ lie in $R$.  One can thus use B\"acklund
transformations and the procedure of Theorem \ref{w0} to construct a
matrix $U(z, \{\mathbf{w}^i_k\}, \{\bm{\zeta}_i\}, \{\bm{\Lambda}_i\})\in B_+(R)$
satisfying the equation
  \begin{equation}\label{formal}\mathcal{A}(z)=U(z,\{\mathbf{w}^i_k\},
    \{\bm{\zeta}_i\}, \{\bm{\Lambda}_i\})\Big(\partial_z+\sum^r_{i=1}\bm{\zeta}_i\check\alpha_i\Big)
    U(z,\{\mathbf{w}^i_k\}, \{\bm{\zeta}_i\}, \{\bm{\Lambda}_i\})^{-1}.
\end{equation}
Let $v(z)$ be the specialization of
$U(z, \{\mathbf{w}^i_k\}, \{\bm{\zeta}_i\}, \{\bm{\Lambda}_i\})$ at the
data for our original $\nabla$.  We then obtain
$A(z)=v(z)(\partial_z+Z^H)v(z)^{-1}$ as desired.

  \end{proof}

  \begin{Thm}\label{Mbethe} There is a one-to-one
    correspondence between the set of nondegenerate $Z^H$-twisted
    Miura $G$-opers with admissible combinatorial data and the set of
    solutions to the $Z^H$-twisted Bethe Ansatz equations for
    ${}^L\fg$ with the same combinatorial data.
  \end{Thm}
  \begin{proof}  This follows immediately from Theorems~\ref{MPbethe} and \ref{thm:MPisM}.
  \end{proof}

  \begin{Rem}
In \cite{Frenkel:2020}, the authors studied a difference equation version of the
$qq$-system involving quantum Wronskians called the \emph{$QQ$-system}.  In this paper, it is shown that there is a
bijection between  twisted
Miura-Pl\"ucker $(G,q)$-opers (with regular semisimple twist
  parameter) and solutions to the
Bethe Ansatz equations for the XXZ model, and this correspondence goes
through the intermediary of
polynomials solutions of the $QQ$-system.  There is an analogue of
$w_0$-genericity in this context, and as for ordinary opers, a
$w_0$-generic Miura-Pl\"ucker $q$-oper is in fact a Miura $q$-oper.
The methods of this paper can be used to prove the $q$-oper analogue
of Theorem~\ref{thm:MPisM}:  a Miura-Pl\"ucker
$q$-opers with admissible combinatorics is a Miura $q$-oper.
  \end{Rem}

\bibliography{biblio}

\end{document}